\documentclass{article}

\usepackage{amsmath,amsthm,amssymb}
\usepackage{enumerate}
\usepackage{graphicx,subfigure}
\usepackage{cite}
\usepackage{mathtools}
\usepackage{authblk}
\usepackage{tikz}
\usepackage{hyperref}
\theoremstyle{plain} 
\newtheorem{thm}{Theorem}[section]
\newtheorem{cor}[thm]{Corollary}

\newtheorem{lem}[thm]{Lemma}
\newtheorem{prop}[thm]{Proposition}
\newtheorem{prob}[thm]{Problem}
\theoremstyle{definition}

\theoremstyle{remark}
\newtheorem{rem}[thm]{Remark}
\newtheorem{ex}[thm]{Example}
\newtheorem*{notation}{Notation}
\theoremstyle{definition} 
\newcommand{\R}{\mathbb{R}}

\pgfdeclarelayer{background}
\pgfdeclarelayer{foreground}
\pgfsetlayers{background,main,foreground}
\tikzstyle{vertex}=[draw,circle,fill=blue!20]
\tikzstyle{small vertex}=[draw,circle,fill=blue!20,inner sep = 0pt,minimum size = .2cm]
\tikzstyle{redvertex} = [draw,vertex, fill=red!40]
\tikzstyle{greenvertex} = [draw,vertex, fill=green!40]
\tikzstyle{labels} = [shape=rectangle, fill = blue!20]
\tikzstyle{edge} = [draw,thick,-]
\tikzstyle{rededge} = [draw,line width=5pt,-,red!40]
\tikzstyle{greenedge} = [draw,line width=5pt,-,green!40]
\tikzset{onslide/.code args={<#1>#2}{%
  \only<#1>{\pgfkeysalso{#2}} % \pgfkeysalso doesn't change the path
}}

\begin{document}

\author[1]{Alvaro Carbonero}
\author[2]{Beth Anne Castellano}
\author[3]{Gary Gordon}
\author[4]{Charles Kulick}
\author[5]{Brittany Ohlinger}
\author[6]{Karie Schmitz}

\affil[1]{Univ. of Waterloo, 
Waterloo, ON, Canada  N2L 3G1}
\affil[2]{Dartmouth College, 
Hanover NH 03755}
\affil[3]{Lafayette College, Easton, PA 18042}
\affil[4]{Univ. California at Santa Barbara,
University of California
Santa Barbara, CA 93106}
\affil[5]{Albright College, 1621 N. 13th Street
Reading, PA 19604}
\affil[6]{Syracuse Univ, 
Syracuse University
Syracuse, NY 13244}

\date{}

\title{Permutations of point sets in $\R^d$}
\maketitle
\renewcommand{\thefootnote}{}
\footnotetext{Research supported by NSF grant DMS-1560222.}

\begin{abstract}  Given a set $S$ consisting of $n$ points in $\R^d$ and one or two vantage points, we study the number of orderings of $S$ induced by measuring the distance (for one vantage point) or the average distance (for two vantage points) from the vantage point(s) to the points of $S$ as the vantage points move through $\R^d.$  With one vantage point, a theorem of Good and Tideman [{\em J. Combin. Theory Ser. A}, \textbf{23}: 34--45, 1977] shows the maximum number of orderings is a sum of unsigned Stirling numbers of the first kind. We show that the  minimum value in all dimensions is $2n-2,$ achieved by $n$ equally spaced points on a line. We investigate special configurations that achieve intermediate numbers of orderings in the one-dimensional and two-dimensional cases. We also treat the case when the points are on the sphere $S^2,$ connecting spherical and planar configurations. We briefly consider an application using weights suggested by an application to social choice theory. We conclude with several open problems that we believe deserve further study.
\end{abstract}

\section{Introduction} \label{S:intro} Suppose $n$ candidates are running for office, and the position of each candidate is measured on two independent issues. Each candidate is then assigned an ordered pair of real numbers, with the first number measuring the candidate's position on the first issue, and the second number measuring the position on the second issue. 

A voter $V$ also has positions on the issues, so $V$ also corresponds to an ordered pair. Thus, we now have $n$ points in the plane (corresponding to the $n$ candidates) and one point $V$ (corresponding to the voter). See Fig.~\ref{F:5pts} for an example where $S$ consists of five points in the plane, with a specific vantage point $V.$

\begin{figure}[htb] %  figure placement: here, top, bottom, or page
   \centering
\includegraphics[width=2in]{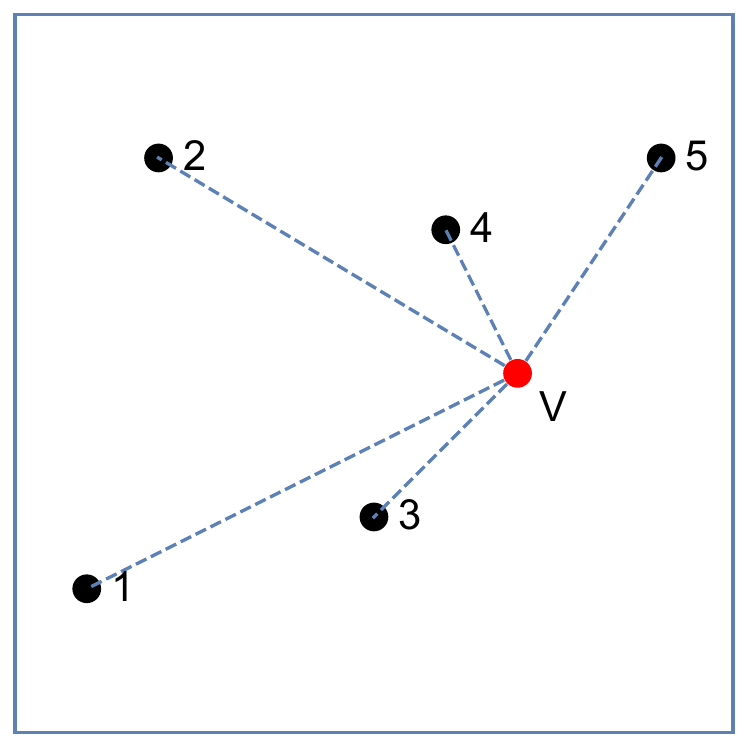} 
 \caption{Five points in the plane, with one vantage point $V.$ Preference ordering: 43521.}
\label{F:5pts}
\end{figure}

Of course, different voters may take different positions on the issues. Each voter ranks the candidates from closest to farthest, using Euclidean distance. This leads to the following problem in discrete geometry:

\begin{prob} For a fixed set $S=\{P_1, P_2, \dots, P_n\}$ of $n$ points in $\mathbb{R}^d$ and a \textit{vantage point} $V,$ generate an ordering of the points of $S$ by measuring the distance from $V$ to each of the points of $S,$ where $P_i<P_j$ in the ordering precisely when $d(V,P_i)<d(V,P_j).$  How many distinct orderings can be produced for different vantage points $V$?
\end{prob}

This problem is the focus of Good and Tideman \cite{gt} and Zaslavsky \cite{z}. In 1977, Good and Tideman determined the maximum possible number of orderings of the $n$ points in $d$ dimensions. In particular, they prove the following theorem.

\begin{thm}\label{T:gt} Suppose $n$ points are situated ``freely'' in $\mathbb{R}^d.$ Then the number of orderings produced as the vantage point moves in $\mathbb{R}^d$ is $$s(n,n)+s(n,n-1)+\cdots + s(n,n-d),$$ where $s(n,k)$ is the (unsigned) Stirling number of the first kind.

\end{thm}

The unsigned Stirling number $s(n,k)$ gives the number of permutations of a set of size $n$ having precisely $k$ cycles. The proof of Theorem~\ref{T:gt}  given in \cite{gt} is inductive. Zaslavsky \cite{z} gives a different proof based on hyperplane arrangements. The connection to hyperplane arrangements arises naturally: Given two points $A$ and $B,$ voters on one side of the perpendicular bisector (a hyperplane in $\mathbb{R}^d$) will prefer $A$ to $B$, and voters on the other side of the hyperplane will have the opposite preference. See Fig.~\ref{F:perp} for an example with $n=4$ in the Euclidean plane. (We ignore voters situated on the hyperplanes; these give rise to ``pseudo-orderings'' in which ties are allowed in the ranking of the candidates.)

\begin{figure}[htb] %  figure placement: here, top, bottom, or page
   \centering
\includegraphics[width=2in]{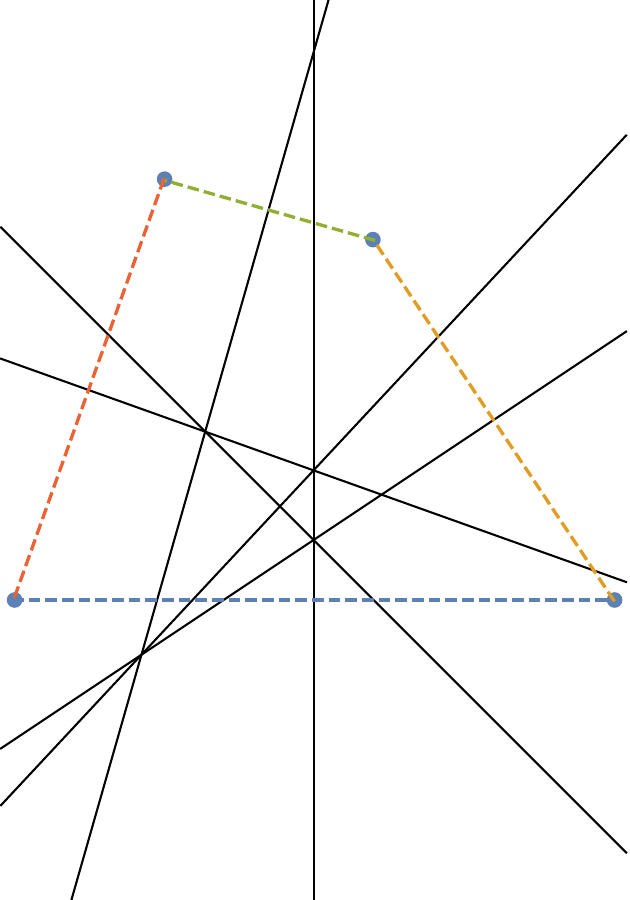} 
 \caption{Perpendicular bisectors partition the plane into  regions, each of which corresponds to a unique ordering of $S.$}
\label{F:perp}
\end{figure}

\begin{rem} Our interpretation of a set of points being ``freely situated'' is that such a configuration produces the maximum possible number of orderings. Zaslavsky points out \cite{z} that it appears to be difficult to say precisely what this condition means geometrically. We also point out that the connection to the unsigned Stirling numbers connects a problem with permutations (counting the number of orderings) with a statistic associated with permutations. But it appears no direct connection is known, i.e., we are unaware of a direct proof of Theorem~\ref{T:gt} that associates permutations to permutations.
\end{rem}

In the plane, the formula for the maximum number of possible orderings is given by a degree 4 polynomial.

\begin{prop}\label{P:1max}
The maximum number of orderings for $n$ points in the plane is $$s(n,n)+s(n,n-1)+s(n,n-2)=\frac{1}{24}\left(3n^4-10n^3+21n^2-14n+24\right).$$
\end{prop}

A direct (non-inductive) proof of this proposition follows from Euler's polyhedron formula $v-e+f=2$ and an analysis of the line arrangement formed by the perpendicular bisectors of all pairs of points. In this context, there is an obvious one-to-one correspondence between orderings and regions of the plane determined by all the perpendicular bisectors. The integer sequence generated by the formula of Prop.~\ref{P:1max} appears in the online encyclopedia of integer sequences \cite{s} (OEIS A308305). Details concerning this approach to the derivation of the formula in Prop.~\ref{P:1max}  are left to the interested reader, and similar arguments appear in our proof of Lemma~\ref{L:sum}.

We can compare the maximum  number of regions determined by $n$ points in the plane from Prop.~\ref{P:1max} to the maximum number of regions produced by an arbitrary collection of ${n \choose 2}$ lines in the plane. Since the maximum number of regions determined by $k$ lines in the plane is ${k \choose 0} + {k \choose 1}+{k \choose 2}=\frac{k^2+k+2}{2},$ we know that ${n \choose 2}$ lines determine at most $\frac{1}{8} \left(n^4-2 n^3+3 n^2-2 n+8\right)$ regions. And, while an arbitrary collection of ${n \choose 2}$ lines in the plane produces more regions than the special collection of ${n\choose 2}$ perpendicular bisecting lines produced in Prop.~\ref{P:1max}, the ratio of these two values approaches 1 since the lead terms in these two formulas are identical. We summarize this observation with the following corollary.

\begin{cor}\label{C:ratio}
Let $M(n)$ be the maximum number of regions determined by the perpendicular bisectors of $n$ points in the plane, and let $L(k)$ be the maximum number of regions formed by a collection of $k$ lines in the plane. Then $$\frac{M(n)}{L\left({n \choose 2}\right)} \to 1.$$
\end{cor}

We introduce some useful notation.

\begin{notation}
\begin{enumerate}
\item Let $S=\{P_1, P_2, \dots, P_n\}$ be a configuration of $n$ points in $\mathbb{R}^d$ dimensions. For a vantage point $V,$ generate an ordering of the points of $S$ by measuring the distance from $V$ to each of the points of $S,$ as above. Let $a_S(n,d)$ be the number of distinct orderings generated when we move a single vantage point $V$ through $\mathbb{R}^d.$ When there is no ambiguity about the dimension or cardinality, we will simply write $a_S$  for the number of orderings generated.

\item Maximum and minimum: Let $M(n,d)$ and $m(n,d)$ be the maximum and minimum values of $a_S(n,d)$ over all possible $S \subset \mathbb{R}^d,$ with $|S|=n.$  When the dimension of our space is clear, we will simply write $M(n)$ and $m(n)$ for the maximum and minimum values.
\end{enumerate}

\end{notation}

From Theorem~\ref{T:gt}, we know $M(n,d)=\sum_{i=0}^d s(n,n-i),$ and it is obvious that $m(n,d) \leq a_S(n,d) \leq M(n,d)$ for any $S\subset \mathbb{R}^d$ with $|S|=n.$ 
Our interest in this problem considers several variations, all of which are easy to motivate. We discuss the generalizations now, along with an outline to the structure of the rest of this paper.

 Let $S$ be a configuration of $n$ points in $d \geq 1$ dimensions. In Section~\ref{S:gaps}, we focus on two questions:
\begin{itemize}
\item Can we determine an exact formula for the minimum value $m(n,d)$?
\item Various special configurations may produce values between the max and the min. Can we fill in these gaps, i.e., for a given integer $k$ satisfying $m(n,d) < k < M(n,d),$ is there a configuration $S \subset \mathbb{R}^d$ with $a_S(n,d)=k$?
\end{itemize}

We determine the minimum $m(n,d)$ in Theorem 2.1.

\smallskip

\noindent \textbf{Theorem \ref{T:min}.}  \textit{The minimum value of $a_S(n,d)$ is  $m(n,d)=2n-2.$ Further, if $n\neq 4,$ this value is achieved if and only if the points of $S$ are equally spaced on a line segment. For $n=4,$ this value is achieved if and only if the points are on a line with $d(P_1,P_2)=d(P_3,P_4).$}

This result is not surprising, but filling in the gaps between the minimum and maximum is harder in dimensions $d>1.$ We show how to find configurations $S$ in the plane close to the minimum where the difference $a_S(n,2)-m(n,2)$ is small, and other configurations close to the maximum where we control the difference $M(n,2)-a_S(n,2).$ Our main tool is a ``free-sum'' lemma (Lemma~\ref{L:sum}) that allows us to compute the number of regions formed by freely placing two configurations $S$ and $T$ together in the plane. Then the number of regions determined by the points in this free sum is completely determined by the sizes of $S$ and $T$ and the number of regions they determine separately. We apply this lemma to special configurations to fill in some of the gaps. A summary of these results appears in Cor. \ref{C:gaps}:

\smallskip

\noindent \textbf{Corollary \ref{C:gaps}.} \textit{Let $n$ be a positive integer and let $k$ be chosen so that $$k \in \left[m(n),\frac{n^2-n+2}{2} \right] \cup \left[M(n)- \frac{n}{2} ,M(n)\right].$$ Then there is a planar point configuration $S$  with $a_S=k.$}

%\noindent We point out, however, that this approach will still leave large intervals where we have no information.

\smallskip

 In Section~\ref{S:sphere}, we show how to convert our formulas for points in the plane to points on a sphere. While we do not consider specific applications in this paper, we believe problems connecting the embedding of points on various surfaces is an interesting topic worth exploring on its own. We compute the maximum in Theorem \ref{SphereMax} and the minimum in Theorem \ref{T:minsphere}. We conclude the section by computing the number of regions for configurations of points corresponding to the five Platonic solids.

\medskip

 We introduce two generalizations in Section~\ref{S:gen}. 
\begin{itemize}
\item When evaluating various candidates,  a specific voter $V$ may care more about some issues than others. We model this situation by having the voter assign non-negative real numbers (\textit{weights}) to each issue to reflect that issue's relative importance to that voter. In Section~\ref{SS:weight}, we will  reduce the problem of computing an ordering using \textit{weighted} preferences to the unweighted case by transforming the point set $S$ in a natural way.

\item What if there are two voters who wish to create a common list of candidates? We call this  the ``Yard sign problem,'' where two people must decide on a common ordering they can agree on. While this problem makes sense for any collection of points in $d$ dimensions, in Section~\ref{SS:2points} we concentrate on points in the plane. Then, given a set $S$ of $n$ points in the plane and \textit{two} vantage points $V_1$ and $V_2,$ we produce an ordering where $P_i<P_j$ if $d(V_1,P_i)+d(V_2,P_i)<d(V_1,P_j)+d(V_2,P_j).$ This is equivalent to using the \textit{average} distance from the two vantage points $V_1$ and $V_2$ to determine the ordering of the points of $S.$
\end{itemize}

In general, it appears quite difficult to determine exact formulas for the maximum and minimum values in this case. Proposition \ref{P:velo} provides an upper bound when the point set is collinear. The proof uses calculus in a novel way and the Fibonacci sequence appears, but much more work is needed before we have a complete understanding of this situation.

\medskip

There are several unsolved problems that deserve further exploration. We list some of these in Section~\ref{S:prob}.

\section{Finding the minimum}\label{S:min}

Let $S \subset \mathbb{R}^d$ with $|S|=n.$ From Good and Tideman \cite{gt}, we know the maximum number of orderings possible is given by $M(n,d)=\sum_{i=0}^ds(n,n-i),$ where $s(n,k)$ is the $k^{th}$ (unsigned) Stirling number of the first kind. We list the  specific polynomials giving the maximum in small dimensions below.

\begin{eqnarray*}
M(n,1) &=&\frac{1}{2} \left(n^2-n+2\right) \\
M(n,2)&=& \frac{1}{24} \left(3 n^4-10 n^3+21 n^2-14 n+24\right) \\
M(n,3)&=& \frac{1}{48} \left(n^6-7 n^5+23 n^4-37 n^3+48 n^2-28 n+48\right)
\end{eqnarray*}

When $n\geq 2d,$ we note that $M(n,d)$ will be a polynomial of degree $2d.$ This follows from analyzing the individual terms contributing to the Stirling number $s(n,n-d).$ In this case, note that there are $(2d-1)!! {n \choose 2d}$ permutations composed of $n-2d$ fixed points and $d$ transpositions, where $(2d-1)!!=(2d-1)(2d-3)(2d-5)\cdots 3\cdot1.$ It is clear this term will generate the highest power of $n.$ (Note that $s(n,1)=(n-1)!,$ so $M(n,d)$ is not a polynomial in $n,$ in general.)

The situation for the minimum value is  simpler. The formula does not depend on the dimension we are working in, and the point configurations achieving the minimum must be collinear and equally spaced, with the exception of the $n=4$ case, where we can drop the requirement on equal spacing.

\begin{thm}\label{T:min}
The minimum value of $a_S(n,d)$ is  $m(n,d)=2n-2.$ Further, if $n\neq 4,$ this value is achieved if and only if the points of $S$ are equally spaced on a line segment. For $n=4,$ this value is achieved if and only if the points are on a line with $d(P_1,P_2)=d(P_3,P_4).$
\end{thm}
\begin{proof} First, we show $m(n,d) \leq 2n-2$ by computing $a_S(n,d)$ when $S$ consists of $n$ equally spaced points on a line. For convenience, assume the points are placed at the first $n$ positive integers along the $x_1$-axis in $\mathbb{R}^d,$ so $P_1=(1,0,\dots,0), P_2 = (2,0,\dots, 0), \dots, P_n=(n,0, \dots,0).$ Then the hyperplane corresponding to the perpendicular bisector of the points $P_i$ and $P_j$ has equation $x_1=\frac{i+j}{2}.$ These $n$ points determine a total of $2n-3$ perpendicular bisecting hyperplanes, namely all hyperplanes with equations $x_1=k/2,$ where $3 \leq k \leq 2n-1.$ These $2n-3$ parallel hyperplanes partition $\mathbb{R}^d$ into $2n-2$ parallel strips. Since we know there is a one-to-one correspondence between orderings and regions, this shows $m(n,d)\leq 2n-2.$

Next, we show that $m(n,d)\geq 2n-2.$  We will use induction on the dimension $d$ to accomplish this. Actually, we will prove a slightly stronger statement, namely that if $S\subset \mathbb{R}^d$ where $d \geq 2,$ then $a_S(n,d)=2n-2$ implies the points of $S$ are collinear. This will then reduce everything to the one-dimensional case.
\begin{itemize}
\item $d=1.$ Let $a_1, a_2, \dots, a_n$ be $n$ real numbers, listed in increasing order. Note that 
$$a_1+a_2<a_1+a_3<\cdots <a_1+a_n<a_2+a_n<a_3+a_n<\cdots < a_{n-1}+a_n.$$
This immediately implies that the $n$ points determine at least $2n-3$ distinct ``perpendicular bisectors'' (which correspond to midpoints in this case). These midpoints partition the line into at least $2n-2$ segments, each of which corresponds to a unique ordering of the points. Thus $m(n,1)\geq 2n-2.$

\item $d=2.$ Suppose $S \subset \mathbb{R}^2$ with $a_S(n,2) \leq 2n-2.$ We will show that the points of $S$ must be collinear by contradiction, so we suppose the points are not collinear. Then, by Ungar's Theorem \cite{u} on slopes, the $n$ points in $S$ determine at least $n-1$ distinct slopes. Since the slope of a segment uniquely determines the slope of its perpendicular bisector, we conclude that there are at least $n-1$ slopes among the collection of perpendicular bisectors of pairs of points of $S.$ 

Since $k$ lines with distinct slopes give rise to $2k$ unbounded regions of the plane, we conclude that the perpendicular bisectors partition the plane into at least $2n-2$ unbounded regions. But there must be at least one bounded region, too. To see this, note that if not, then the $n-1$ perpendicular bisectors would all intersect at a common point $C.$ It follows immediately that the points of $S$ lie on a circle. But $n$ points on a circle determine at least $n$ slopes (this is due to Erd\H{o}s and Steinberg --- see \cite{dbe,es}), giving rise to at least $n$ perpendicular bisectors. In this case, these  bisectors partition $\mathbb{R}^2$ into at least $2n$ unbounded regions, so $a_S(n,2) \geq 2n,$ contradicting our assumption that  $a_S(n,2) \leq 2n-2.$

We now know there is at least one bounded region in the partition determined by the collection of perpendicular bisectors, giving a total of at least $2n-1$ regions, again contradicting the assumption that  $a_S(n,2) \leq 2n-2.$

We conclude that if the points of $S$ are not collinear, then $a_S(n,2)>2n-2.$

\item $d>2.$ Let $S$ be a collection of $n$ points in $\mathbb{R}^d.$ If the points lie in some hyperplane, then we have $a_S(n,d) > 2n-2$ by the induction hypothesis. So we now assume the points of $S$ do not lie in any $k$-dimensional hyperplane for $k<d.$ 

Now project the points of $S$ onto a $3$-dimensional hyperplane $H'$ in $\mathbb{R}^d$ such that:
\begin{enumerate}
\item [a.] no two points in $S$ are projected onto the same point in $H'$, and
\item [b.] the projected points do not lie on a 2-dimensional plane.
\end{enumerate}
This is always possible. Let $P_i'$ denote the projection of $P_i$ onto $H'.$ Then, by Theorem 1.1 of \cite{pps}, the points $P_1',P_2', \dots, P_n'$ determine at least $2n-5$ different directions. This implies that the points $P_1, P_2, \dots, P_n$ in $\mathbb{R}^d$ also determine at least $2n-5$ distinct directions because projection is a linear transformation, and so it preserves parallelism. (If the directions for two vectors $P_i'-P_j'$ and $P_k'-P_l'$ are different in $H'$, then the vectors $P_i-P_j$ and $P_k-P_l$ could not be parallel.)

Now each direction vector $P_i-P_j$ in $\mathbb{R}^d$ is the normal vector for the perpendicular bisecting hyperplane determined by $P_i$ and $P_j.$ Hence, we must have at least $2n-5$ distinct, pairwise non-parallel, perpendicular bisecting hyperplanes determined by the points of $S$ in $\mathbb{R}^d.$ Now projecting from $\mathbb{R}^3$ to $\mathbb{R}^2$ gives us at least $2(2n-5)$ unbounded regions of $\mathbb{R}^2,$ each corresponding to a unique ordering of the points of $S.$ Hence, $a_S(n,d)>2n-2.$

\end{itemize}
Now assume $n>4.$ We must show that if $S$ minimizes the number of regions, then $S$ consists of $n$ equally spaced points on a line. By the proof of the first part of this theorem, we know that the points of $S$ must be collinear. As above, let $a_1, a_2, \dots, a_n$ be $n$ real numbers, listed in increasing order, which represent the location of the $n$ points of $S$ in $\mathbb{R}.$ 

Now consider the first five points of $S.$ These points determine at most ten midpoints. Since $a_S(n,1)=2n-2,$ the first five points must generate only seven distinct midpoints. But we know (as in the proof above) that the seven sums $$a_1+a_2<a_1+a_3<a_1+a_4<a_1+a_5<a_2+a_5<a_3+a_5<a_4+a_5$$ are all distinct. This tells us each of the remaining sums $a_2+a_3, a_2+a_4,$ and $a_3+a_4$ must be equal to one of the other sums already listed. It is easy to see that the only possibilities are 
\begin{eqnarray}
a_2+a_3 &=& a_1+a_4, \\
a_2+a_4 &=& a_1+a_5, \\
a_3+a_4&=& a_2+a_5.
\end{eqnarray}
First, rewrite equation (1) as $a_4-a_3=a_2-a_1.$ Now subtracting equation (1) from equation (2) gives $a_5-a_4=a_4-a_3,$ and subtracting equation (2) from equation (3) gives $a_3-a_2=a_2-a_1.$  Putting the pieces together gives  $$a_2-a_1=a_3-a_2=a_4-a_3=a_5-a_4,$$ i.e., the points are equally spaced.

Now repeat this argument for the five points $a_2, a_3, a_4, a_5, a_6,$ then $a_3,a_4,a_5,a_6,a_7,$ continuing this process until we exhaust $S.$ We conclude that if $n>4$ and $a_S(n,d)=2n-2,$ then the points of $S$ are collinear and equally spaced.

Finally, when $n=4,$ it is straightforward to check that $a_2-a_1=a_4-a_3,$ but $a_3-a_2$ need not be equal to this common value. For instance, the points $a_1=1, a_2=2, a_3=4, a_4=5$ determine five distinct midpoints, and so achieve the minimum $m(4,1)=6$ regions.

\end{proof}
We conclude this section with some brief comments.
\begin{enumerate}
\item [A.] The appearance of Ungar's Theorem \cite{u} and the direction theorem of Pach, Pinchasi, and  Sharir \cite{pps} provides a direct connection between our problem and two classic problems from discrete geometry. In fact, Ungar's remarkable proof is featured in \cite{az} as a model for a ``book proof'' in the style of Erd\H{o}s. The introduction to \cite{pps} includes more historical information about these problems, and \cite{j} is an excellent resource for the slope problem in the plane.

\item[B.] We can avoid using Ungar's theorem in the $d=2$ case in our proof. To see this, note that $n$ non-collinear points in the plane determine at least $n$ lines, from Erd\H{o}s and Steinberg \cite{es}. These $n$ lines must give rise to $2n$ unbounded regions, unless the lines are all parallel to each other. But this only happens if the points are collinear. 

\item[C.] When $d>2,$ we can avoid using the direction theorem of  \cite{pps} as follows. If $S$ is not collinear, we can project $S$ to $R^2$ and get a non-collinear set of $n$ points in the plane. By Ungar's theorem we can find $n-1$ pairs of points that determine distinct directions. The corresponding perpendicular
bisectors in $R^d$ for the same set of $n-1$ pairs of points determine $n-1$ hyperplanes, no two of which are parallel. Then these hyperplanes partition $R^d$ into
at least $2n-2$ regions.

\item [D.] If we insist that our point set $S \subset \mathbb{R}^d$ is not contained in any lower dimensional hyperplanes, then the question about the minimum number of orderings is open. This  appears as Problem~\ref{Pr:min}  in Section~\ref{S:prob}, where we list several open problems.
\end{enumerate}

\section{Filling in the gaps between the maximum and the minimum}\label{S:gaps} 
In this section, we consider the problem of finding ``intermediate'' configurations that produce numbers of orderings strictly between the maximum and minimum values. This is completely straightforward in dimension 1, but is already quite difficult when $d=2.$ 

\subsection{Points on a line}\label{SS:line} We begin with points on a line. We will show that, for all $k$ with $m(n,1) < k <M(n,1),$ there is a point configuration $S_k$ satisfying $a_{S_k}(n,1)=k.$ 

\begin{thm}\label{T:gap1}
Let $n\geq 1$ be given, and let $k$ satisfy $2n-2 \leq k \leq \frac{n^2-n+2}{2}.$ Then there is a configuration of points $S_k$ with $a_{S_k}(n,1)=k.$
\end{thm}
\begin{proof}
Our proof is constructive: we show how to create such a configuration by starting with a configuration achieving the minimum value, then modifying that configuration to increase the number of regions determined by the midpoints in a predictable way. As before, we represent our linear point configuration by specifying $n$ real numbers in increasing order. We let $S_0=\{1,2,\dots, n\}$ be our initial configuration, and remark that $S_0$ achieves the minimum value of $2n-2.$ The proof proceeds by rounds.

\smallskip

\textbf{Round 1}.  In the first round, we keep the first $n-1$ points fixed and we move the last point successively to the values $n+1,$ then $n+2,$ and so on, continuing until we reach $2n-3.$ We call these configurations $S^{(1)}_1, S^{(1)}_2, \dots, S^{(1)}_{n-3},$ so $S^{(1)}_t=\{1,2,\dots, n-1,n+t\},$ with $1 \leq t \leq n-3.$ (The superscript indicates the round.) We now show that each new configuration will have exactly one more midpoint than its immediate predecessor. 

To see this, we note that there are $2n-5$ midpoints generated by the first $n-1$ points of $S^{(1)}_t$ occurring at the values $\frac{i}{2},$ where $3 \leq i \leq 2n-3,$  as well as an additional $t+2$ midpoints formed from the pairs $r$ and $n+t,$ where $n-(t+2) \leq r \leq n-1.$  This produces a total of $2n+t-3$ midpoints, so the number of regions generated (which are intervals here) is $2n+t-2.$

We conclude that fixing the first $n-1$ points and moving the right-most point through the values $n, n+1, \dots, 2n-3$  produces configurations partitioning $\mathbb{R}$ into exactly $k$ intervals for $2n-2\leq k \leq 3n-5.$  (The reader can check that values of $t$ larger than $n-3$ do not introduce any additional midpoints.)

\smallskip
\textbf{Round 2}. In the second round,  for $1 \leq t \leq n-4,$ we fix the first $n-2$ points, we move the penultimate point to $n+t-1,$ and we move the last point to $2^n.$ Thus $S^{(2)}_t=\{1,2,\dots,n-2,n+t-1,2^n\}.$ We now count the number of distinct midpoints:
\begin{enumerate}
\item If $1 \leq i,j \leq n-2,$ then the midpoints of $i$ and $j$ produce a total of $2n-7$ midpoints corresponding to the values $\frac{i}{2},$ where $3 \leq i \leq 2n-5.$
\item If $n-(t+3) \leq i \leq n-2,$ then pairing the points $i$ and $n+t-1$ produces another $t+2$ new midpoints at the values $\frac{r}{2},$ where $2n-4 \leq r \leq 2n+t-3.$
\item If $1 \leq i \leq n-2$ or $i=n+t-1,$ then pairing $i$ and $2^n$ (the largest point in $S^{(2)}_{t}$) yields another $n-1$ new midpoints.% at the half-integers $\frac{r}{2},$ where $2n+t \leq r \leq 3n+t-3,$ and one additional midpoint $\frac{3n+2t-2}{2}.$
\end{enumerate}
These three cases give a total of $3n+t-6$ midpoints, generating $3n+t-5$ intervals in our partition. We conclude that fixing the first $n-2$ points and moving the two right-most points as above  produces configurations partitioning $\mathbb{R}$ into exactly $k$ intervals for $3n-4\leq k \leq 4n-9.$

We continue in this manner, adding new rounds to fill in all the gaps. This will give us a total of $n-3$ rounds that produce $1+2+\cdots+ (n-3)=M(n,1)-m(n,1)$ distinct numbers of orderings, filling in all the gaps between the minimum and the maximum. We summarize the procedure in Table~\ref{Ta:rounds}.

\begin{table}
\caption{Intermediate configurations that partition the reals into exactly $k$ disjoint intervals, for $2n-1 \leq k \leq \frac{n^2-n+2}{2}.$ We set $b_m=\frac{m^2+3m-2}{2}$ and $c_m=\frac{m^2+5m+4}{2}$ for convenience.}
\begin{center}
\begin{tabular}{|c|c|c|c|} \hline
Round & Configuration & Range for $t$ & Range for $k$ \\ \hline
1&$\{1,2,\dots, n-1,n+t\}$ & $[1,n-3]$ & $[2n-1,3n-5]$ \\ \hline
 2& $\{1,2,\dots,n-2,(n-1)+t,2^n\}$ & $[1, n-4]$ &  $[3n-4,4n-9]$ \\ \hline
3& $\{1,2,\dots, n-3, (n-2)+t, 2^{n-1},2^n\}$ & $[1, n-5]$ &  $[4n-8,5n-14]$ \\ \hline
 \vdots & \vdots & \vdots & \vdots \\ \hline
 $m$ & $\{1,2,\dots, n-m,(n-m+1)+t,$ & $[1, n-(m+2)]$ & $[(m+1)n-b_m,(m+2)n-c_m]$ \\
 & $2^{n-m+2},\dots ,2^{n-1},2^n\}$ & & \\ \hline
  \vdots & \vdots & \vdots & \vdots \\ \hline
  $n-3$ & $\{1,2,3, 4+t,2^4, \dots, 2^n\}$ & $t=1$ & $k=\frac{n^2-n+2}{2}$ \\ \hline 
\end{tabular}
\end{center}
\label{Ta:rounds}
\end{table}%

\end{proof}

We remark that Theorem~\ref{T:gap1} has an attractive reformulation as a sum-set problem.

\begin{quote}
\textbf{Sum-set Problem}. Let $n$ be given and let $k$ satisfy $2n-3 \leq k \leq \frac{n^2-n}{2}.$ Then there is a collection of integers $a_1< a_2< \dots< a_n$ such that the number of distinct  sums $a_i+a_j$ (where $i\neq j$) is exactly $k.$ 
\end{quote}

We also point out that, while our initial configuration is uniquely determined (up to the position of the first point and scaling), there are many choices for our final configuration achieving the maximum number of midpoints. We used powers of 2 in our proof, but many other procedures will work. 

\subsection{Points in the plane} \label{SS:plane} It is much more difficult to produce configurations with a specified number of orderings between the minimum and maximum in the plane. In fact, unlike the situation described in Theorem~\ref{T:gap1}, not every intermediate value is achievable. For example, when $n=3,$ we find $m(3,2)=4$ and $M(3,2)=6,$ but it is easy to check that there is no 3-point configuration whose perpendicular bisectors determine exactly 5 regions. When $n=4,$ we know $m(4,2)=6$ and $M(4,2)=18,$ but it turns out that  there are no four-point configurations with $9, 11, 13, 14,$ or $15$ regions formed by the perpendicular bisectors. (We can prove that the missing values cannot be achieved by any configuration of four points  by examining a few classes of 4-point configurations. We omit the details and note that similar arguments for $n>4$ points rapidly become unwieldy.)  
%\begin{figure}[htb] %  figure placement: here, top, bottom, or page
%   \centering
%\includegraphics[width=2.5in]{4pts.png} 
%
%{\footnotesize 6 \hskip.05in 7 \hskip.05in 8 \hskip.3in $\cdots$ \hskip.45in 16  17  18}
%
% \caption{For $n=4$ points, we can achieve $6, 7, 8, 10, 12, 16, 17,$ or $18$ regions. The black squares are achievable, and the white squares are not.}
%\label{F:gap}
%\end{figure}

%\textcolor{red}{Need more pictures here! Or a table with the values Charles' computer program found.}
%
%\bigskip

When $5\leq n \leq 8,$ a computer search was used to produce configurations with intermediate values. But values not appearing have not been explicitly ruled out by any mathematical arguments. Thus, the percentages given in Table~\ref{Ta:gaps} are lower bounds for the actual percentages of achievable values.%Table~\ref{Ta:gaps} gives the results of a computer search for configurations of at most 8 points achieving intermediate numbers of regions.

\begin{table}[htp]
\caption{The minimum,  maximum, and  percentage of achievable orderings produced by a computer search for $3\leq n\leq 8.$}
\begin{center}
\begin{tabular}{|c|c|c|c|c|c|c|} \hline
$n$ & 3 & 4 & 5 & 6 & 7 & 8  \\ \hline 
Min & 4 & 6 & 8 & 10 & 12 & 14  \\ \hline
Max &  6 & 18 & 46 & 101 & 197 & 351  \\ \hline 
Percent & 66.7\% & 61.53 & 61.53\% & 46.74\% & 52.15\% & 58.88\%  \\ \hline
\end{tabular}
\end{center}
\label{Ta:gaps}
\end{table}%

Although we cannot determine exactly what values between the minimum and maximum are achievable by specific configurations for arbitrary $n,$ we can construct configurations that achieve values near the minimum and also near the maximum. These constructions are highlighted in this section. We write $a_S(n), m(n),$ and $M(n)$   instead of $a_S(n,2), m(n,2),$ and $M(n,2)$ (respectively) throughout the remainder of this section since all of our configurations are planar.  When the number of points is implicitly given, we will abbreviate this further, simply writing $a_S$ instead of $a_S(n).$

The strategy for creating configurations that generate intermediate values for $a_S(n)$ is straightforward. 
We will begin with a configuration of $n$ points that achieves the maximum number of orderings, then introduce generalized trapezoidal configurations in Lemma~\ref{L:trap} that reduce the number of regions in a predictable way. The key idea is embedded in Lemma~\ref{L:sum}, which computes the number of regions in the ``free-sum'' of two configurations solely in terms of the sizes of the configurations and the number of regions each piece determines separately. 

To accomplish this, we first construct a graph as follows. Let $S$ be a finite configuration of points in the plane, and let $\mathcal{L}(S)$ be the collection of perpendicular bisectors determined by pairs of points from $S.$ Then draw a circle large enough so that all the points of intersection formed by the lines of $\mathcal{L}(S)$ are inside the circle. We form a graph $G(S)$ whose vertices are the intersection points of the lines of $\mathcal{L}(S)$ with each other and also with the bounding circle. The edges of $G(S)$ will be the (finite length) line segments determined by the intersecting lines, along with the segments formed along the bounding circle. (It is not necessary to introduce the bounding circle, but it simplifies calculations.) See Fig.~\ref{F:euler} for an example. 

\begin{figure}[h] %  figure placement: here, top, bottom, or page
   \centering
\includegraphics[width=3in]{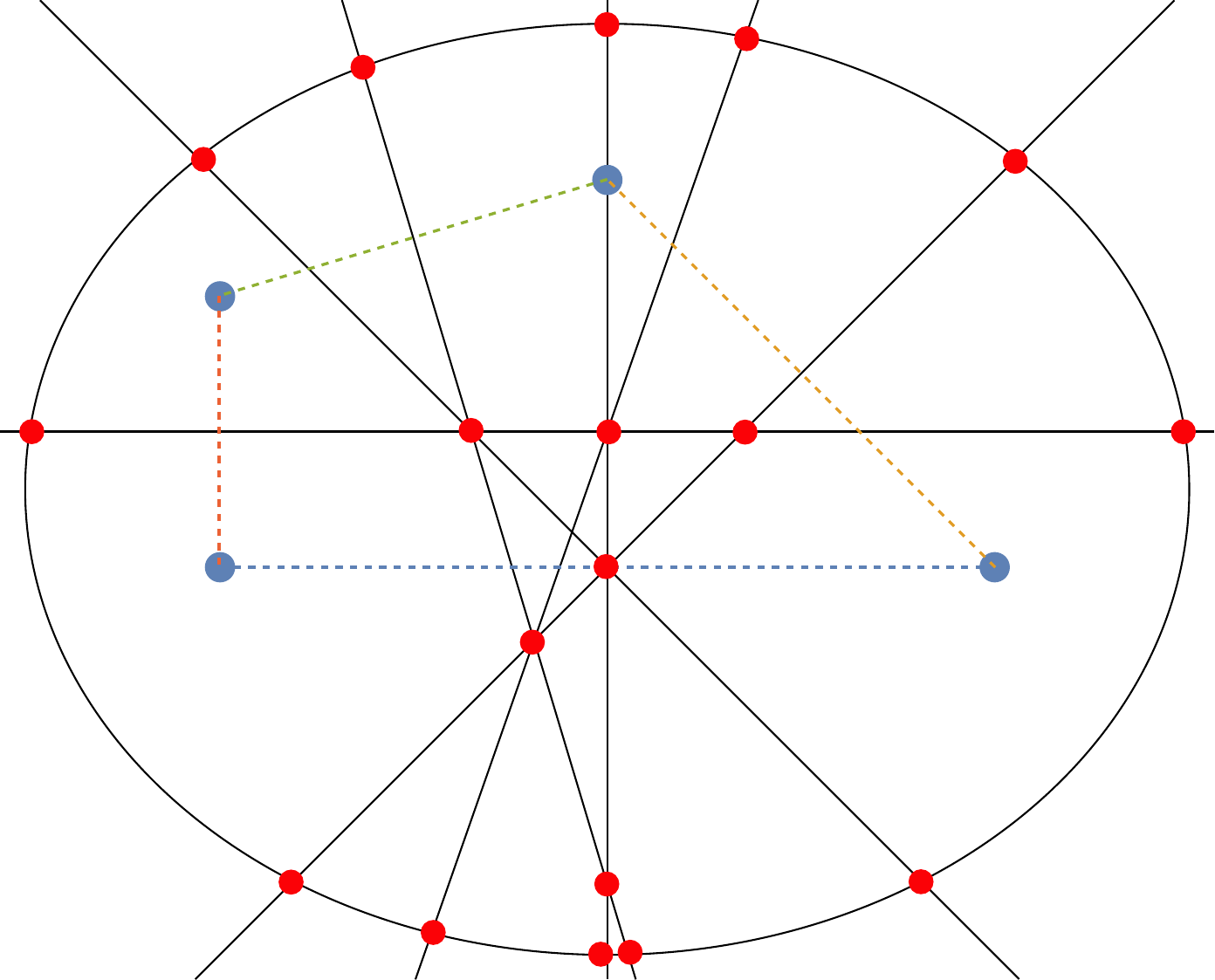}
 \caption{The graph associated with a configuration of  four (blue) points in the plane. The vertices of the graph are shown in red.}
\label{F:euler}
\end{figure}

We can now state and prove a ``free-sum'' lemma. Suppose $S$ and $T$ are disjoint point sets placed freely in the plane, so no perpendicular bisector determined by the points of $S$ is parallel to a perpendicular bisector determined by the points of $T.$  Then we can compute the number of regions for the point set  $S\cup T$ in terms of the original region counts for $S$ and $T$ and the sizes of $S$ and $T.$ Our tool is Euler's famous polyhedral formula $v-e+f=1$ applied to planar graphs (where we ignore the unbounded region). %(Since all of our configurations are planar in this section, we write $a_S$ instead of $a_S(s,2),$ and so on.) 

\begin{lem}\label{L:sum}
Suppose $S$ and $T$ are disjoint point sets in the plane with $|S|=s$ and $|T|=t.$ Assume that no perpendicular bisector determined by the points of $S$ is parallel to a perpendicular bisector determined by the points of $T.$ Then 
$$a_{S\cup T}=a_S+a_T+\frac{1}{4} \left(3 s^2 t^2+2 s^3 t-5 s^2 t+2 s t^3-5 s t^2+7 s t-4\right).$$
\end{lem}
\begin{proof}
We count the number of vertices of all degrees in the graph $G(S\cup T)$ associated to the  configuration $S\cup T.$ We will then use Euler's relation $f=e-v+1$ to compute the number of regions in $S\cup T.$ We write $v(S)$ and $e(S)$ for the number of vertices and edges in the  graph $G(S),$ and similarly for $G(T).$ We first compute the number $d_k$ of new vertices of degree $k$ for all $k.$ 
%The number of new vertices of degree $k$ (i.e., vertices of degree $k$ in $G(S\cup T)$ that are not vertices of $G(S)$ or $G(T)$) is denoted $u(k).$ 
\begin{enumerate}
%\item ``Old'' vertices: Every vertex in the graphs $G(S)$ and $G(T)$ remains a vertex in $G(S\cup T).$ This gives a total of $\sum_{k>0}v_S(k)+\sum_{k>0}v_T(k)$ old vertices.
\item New vertices of degree 3: Note that $S\cup T$ has $st$ new perpendicular bisectors obtained by choosing one point from $S$ and one point from $T.$ Each of these new lines will meet our bounding circle in 2 points, This gives us a total of $d_3=2st$ new vertices of degree 3 in $G(S\cup T).$
\item New vertices of degree 4: We create new vertices of degree 4 whenever we choose 4 points from $S\cup T$ not contained entirely in $S$ or $T.$  Then each such subset of 4 points generates three vertices of degree 4 in $G(S\cup T).$ This gives $$d_4=3s{t \choose 3}+3{s\choose 2}{t\choose 2}+3t{s\choose 3}$$ new vertices of degree 4 in $G(S\cup T).$ (Note that this calculation is valid whether or not points selected in $S$ or $T$ are collinear.)

\item New vertices of degree 6: We create new vertices of degree 6 whenever we create a new triangle using points from both $S$ and $T.$ There are $$d_6=s{t \choose 2}+t{s\choose 2}$$ ways to create new triangles, and each such triangle gives a new vertex of degree 6 in $G(S\cup T).$

\item New vertices of higher degree: There are no new vertices of degree $k>6$ by the definition of $S\cup T.$ 
\end{enumerate}
Then the total number of  vertices in $G(S\cup T)$ is $v(S)+v(T)+d_3+d_4+d_6.$ To find the number of edges, we use $$2e(G(S\cup T))=2e(S)+2e(T)+3d_3+4d_4+6d_6.$$
Then 
\begin{eqnarray*}
r(G(S\cup T))&=& e(S)+e(T)+(3d_3+4d_4+6d_6)/2-(v(S)+v(T)+d_3+d_4+d_6)+1 \\
&=&(e(S)-v(S)+1)+(e(T)-v(T)+1))+(d_3+2d_4+4d_6)/2-1\\
&=& r(S)+r(T)+\frac{1}{4} \left(3 s^2 t^2+2 s^3 t-5 s^2 t+2 s t^3-5 s t^2+7 s t-4\right).
\end{eqnarray*}

\end{proof}

When $S$ and $T$ are both free sets (and so achieve the maximum values possible), we can verify that the formula given in Lemma~\ref{L:sum} agrees with our formula for $M(n,2)$ from Prop.~\ref{P:1max}. Let $|S|=k$ and $|T|=n-k.$ Then
\begin{eqnarray*}
a_{S\cup T} &=& M(k)+M(n-k)+\frac{1}{4} \left(k^2 \left(-3 n^2+5 n-7\right)+2 k^3 n-k^4+k \left(2 n^2-5 n+7\right) n-4\right) \\
&=& \frac{1}{24} \left(3 n^4-10 n^3+21 n^2-14 n+24\right)\\
&=&M(n).
\end{eqnarray*}

Our next lemma shows how to create a generalized trapezoid configuration that will allow us to reduce the number of regions from the maximum in a predictable way.

\begin{lem}\label{L:trap}[Generalized trapezoid gadget] 
Let $U=\{P_1,P_2, \dots, P_{2k}\} \subset \mathbb{R}^2$ be a collection of $2k$ points in the plane in free position, where $k\geq 2,$ with the following exceptional pairs of parallel lines: 
$$\overline{P_1P_2} \| \overline{P_3P_4}, \hskip.2in \overline{P_1P_3} \| \overline{P_5P_6}, \hskip.2in \overline{P_1P_4} \| \overline{P_7P_8}, \hskip.2in \cdots \hskip.2in \overline{P_1P_k} \| \overline{P_{2k-1}P_{2k}}. $$ 
Then $a_U=M(2k)-k+1.$
\end{lem}
 \begin{proof}
Note that, compared with a configuration that produces the maximum possible number of regions,  each parallel pair reduces the number of vertices of degree 4 by 1, but does not change the number of vertices of degree 3 or 6. Then each such parallel pair reduces the total number of vertices by 1 and the total number of edges by 2. Since there are $k-1$ parallel pairs, we have $a_U=M(2k)-k+1.$
\begin{figure}[h]
   \centering
\includegraphics[width=3in]{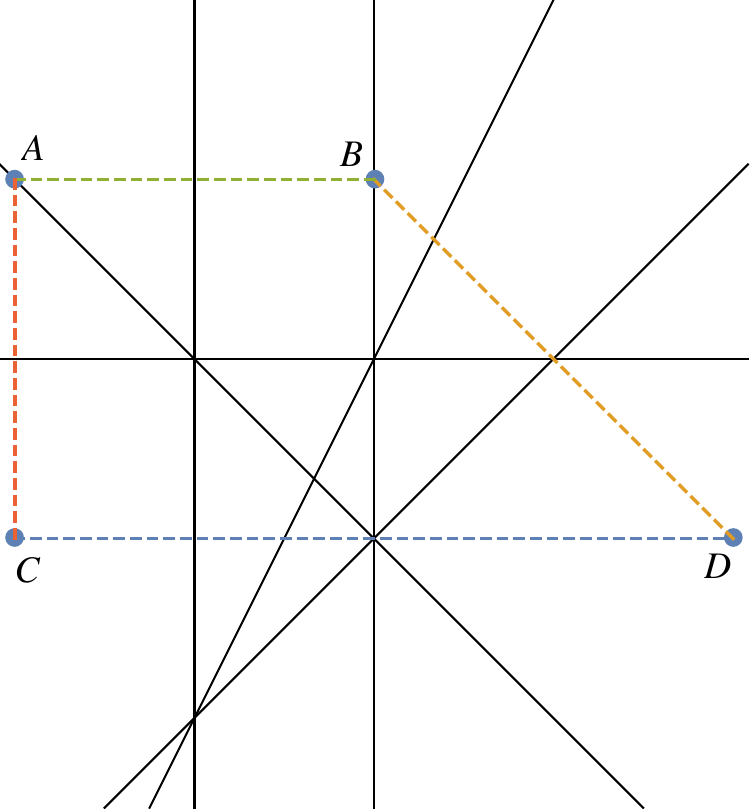} 
 \caption{Lines $\overline{AB}$ and $\overline{CD}$ are parallel, reducing the maximum  number of regions determined by the perpendicular bisectors by 1.}
\label{F:trap}
\end{figure}

 \end{proof}
 
We remark that the formula above remains valid when $k=1,$ but we will not need this fact. When $k=2,$ the configuration is a trapezoid, as in Fig.~\ref{F:trap}.  We now use Lemmas \ref{L:sum} and \ref{L:trap} to get our first gap-filling result.
 
 \begin{thm}\label{T:gap2}
 Let $k\leq \frac{n}{2}.$ Then there is an $n$-point planar point configuration $S$ with $a_S=M(n)-k.$
 \end{thm}
\begin{proof} Let $U$ be a configuration of $2k+2$ points with $k$ parallel pairs, as in Lemma~\ref{L:trap}, and let $V$ be a configuration of $n-2k-2$ points that achieves the maximum possible number of regions, so $a_V=M(n-2k-2). $ By Lemma~\ref{L:trap}, we know $a_U=M(2k+2)-k.$ Now set $S=U\cup V,$ where $U$ and $V$ are placed in the plane so that no additional incidences or parallels are produced.  Then, by Lemma~\ref{L:sum}, we have $$a_S=M(2k+2)+M(n-2k-2)-k+g(2k+2,n-2k-2),$$ where $g(s,t)=\frac{1}{4} \left(3 s^2 t^2+2 s^3 t-5 s^2 t+2 s t^3-5 s t^2+7 s t-4\right)$ is the function appearing in Lemma~\ref{L:sum} that gives the number of new regions in a free sum.

Then, expanding everything and simplifying, we have $a_S=M(n)-k.$ (The calculations are straightforward, but they are easiest to do using a computer program that handles algebra. We used  Mathematica.)

\end{proof} 
 
 Theorem~\ref{T:gap2} tells us that we can achieve any value near the maximum, i.e., given a positive integer $n$ and $0\leq k \leq \frac{n}{2},$ there is a configuration $S_k$ with $a_{S_k}=M(n)-k.$ For region counts near the minimum, we can use Theorem~\ref{T:gap1}. Combining these two results gives us the next result.

\begin{cor}\label{C:gaps}
Let $n$ be a positive integer and let $k$ be chosen so that $$k \in \left[m(n),\frac{n^2-n+2}{2} \right] \cup \left[M(n)- \frac{n}{2} ,M(n)\right].$$ Then there is a planar point configuration $S$  with $a_S=k.$
\end{cor}

Thus, we know there are configurations of $n$ points whose perpendicular bisectors form exactly $k$ regions, where $2n-2\leq k \leq (n^2-n+2)/2$ or $\frac{1}{24} \left(3 n^4-10 n^3+21 n^2-26 n+24\right)\leq k \leq \frac{1}{24} \left(3 n^4-10 n^3+21 n^2-14 n+24\right).$

 We can also find configurations $S$ with $a_S$ not in the ranges specified in Corollary~\ref{C:gaps}. The next two results give examples of such configurations.
 
 \begin{prop}\label{P:par} [Parallel lines gadget]
 Let $C$ be a configuration consisting of $m$ points in free position and $k\cdot l$ points distributed onto $l$ lines, each with $k$ points arranged so that all of the perpendicular bisectors generated are distinct and no two are  parallel. Let $n=m+kl$ be the total number of points of $C.$ Then 
\begin{enumerate}
\item  \begin{eqnarray*}
a_C(n)&=&\left(18 k^2 l^2 m^2+12 k^3 l^3 m-30 k^2 l^2 m+3 k^4 l^4-10 k^3 l^3+21 k^2 l^2-3 k^4 l+10 k^3 l-9 k^2 l\right.\\
& &\left.+12 k l m^3-30 k l m^2+42 k l m-12 k l+3 m^4-10 m^3+21 m^2-14 m+24\right)/24.
\end{eqnarray*}
\item $a_C(n)=M(n)-k s(k,k-2),$ where $s(k,k-2)$ is the unsigned Stirling number of the first kind.
\end{enumerate}

  \end{prop}
  
\begin{proof} Let $U$ be a collection of $n-kl$ points in free position in the plane, and let $V$ be composed of $kl$ points distributed on $l$ pairwise non-parallel lines, with $k$ points per line, as in the statement of the proposition. Then $C=U\cup V.$ We know $a_U(n-kl)=M(n-kl).$ It is straightforward to compute $a_V(kl)$ by counting the number of vertices of degree 3, 4, and 6 in the associated graph $G(V)$ and using Euler's formula. This gives 
$$a_V(k l)=\frac{1}{24} \left(3 k^4 l^4-10 k^3 l^3+21 k^2 l^2-3 k^4 l+10 k^3 l-9 k^2 l-12 k l+24\right).$$
The formula for $a_C(n)$ now follows from Lemma~\ref{L:sum}, completing part 1. 

For part 2, note that $s(k,k-2)=2{k \choose 3}+3{k \choose 4}.$ The rest of the calculation is completely straightforward.

\end{proof}

We omit the proof of the next result, which uses Lemma~\ref{L:sum} and the fact that $m$ points placed freely on a circle generate $2{m \choose 2}$ regions (all of which are unbounded).

\begin{prop} \label{P:cir} [Circle gadget] Let $S$ be a configuration consisting of $n-k$ points in free position and $k$ points placed on a circle so that all of the perpendicular bisectors generated are distinct and no two are  parallel. Then $$a_S=\frac{1}{24} \left(3 n^4-10 n^3+21 n^2-14 n -3 k^4+10 k^3+3 k^2-10 k\right).$$ Further, $a_S=M(n)-M(k)+k(k-1).$  
%\begin{eqnarray*}
%M(n,2)-a_S(n,2)&=& \frac{1}{24} \left(3 k^4-10 k^3-3 k^2+10 k+24\right) \\
%&=&M(k,2)-k(k-1). \\
%\end{eqnarray*}
\end{prop}

We can rewrite the last equation  from Prop.~\ref{P:cir} as $M(n)-a_S=M(k)-k(k-1).$ Then we can  interpret both sides of this equation combinatorially. The left-hand side is the number of regions ``lost'' from a free configuration  by placing the $k$ points on a circle, and the right-hand side is the number of \textit{bounded} regions in a maximum configuration of $k$ points. The sequence $M(k)-k(k-1)$  appears as sequence A001701 in \cite{s}, although the interpretation in terms of lost regions appears to be new. We state this result as a corollary.

\begin{cor}\label{C:cir} Let $S$ be a configuration consisting of $n-k$ points in free position and $k$ points placed on a circle so that all of the perpendicular bisectors generated are distinct and no two are  parallel. Then  $M(n)-a_S=M(k)-k(k-1),$ i.e., the number of regions lost by placing $k$ points on a circle coincides with the number of bounded regions determined by the perpendicular bisectors in a maximum configuration of $k$ points.  
\end{cor}

Finally, we remark that we can combine the various configurations using the free sum operation, filling in other sequences between the minimum and maximum. In particular, the gadgets of Prop.~\ref{P:par} and \ref{P:cir} can be used to create configurations that achieve values not covered in Theorem~\ref{T:gap2}. We leave such explorations to the interested reader.

\section{Spherical configurations}\label{S:sphere}
We next consider a generalization from the plane to $S^2$, the 2-sphere, where all of our points lie on the surface of the sphere and distance is measured using geodesics. An attractive illustration concerns the placement of cell phone towers. If $n$ cell towers are placed on the surface of the earth, a user can measure the distance from their position to each tower, generating an ordering of the towers. As before, we are concerned with the number of different orderings of those towers experienced by various users around the earth.

We will be able to translate results for points in the plane from Section~\ref{SS:plane} to the sphere, including formulas for the maximum, minimum, and known intermediary ranges on the sphere. We also point out particular configurations that reveal how our method of translation from the plane to the sphere would theoretically fall short of a complete account of the achievable intermediary values on the sphere.

%We begin by providing a brief overview of the properties of spherical geometry that will be relevant to our work. 

%\begin{enumerate}
%    \item On the surface of the sphere, the set of all points equidistant from any two points is a great circle.
%    \begin{defn}
%    A \textit{great circle} is the intersection of the sphere and a plane that passes through the sphere's center.
%    \end{defn}
%    
%    \item Two distinct great circles intersect at exactly two antipodal points. In other words, no two great circles can be parallel.
%    \begin{defn}
%    Two points on the sphere are \textit{antipodal} if the line connecting them passes through the sphere's center.
%    \end{defn}
%\end{enumerate}

In this context, the perpendicular bisector determined by two points in the plane is replaced by a bisecting great circle that divides the sphere into two hemispheres (see Figure \ref{fig:2 pt sphere}). The planar graph formed by all of the pairwise bisecting great circles for a given set of points divides the surface of the sphere into a number of regions. As in the plane, these regions each correspond to a unique ordering of the set of points. We point out that every configuration of great circles on a sphere produces a graph whose vertices are the points of intersection of the great circles. Antipodal symmetry implies that each vertex, edge, and region of this graph has a ``mirror image'' in the graph, so the  number of vertices, the number of edges, and the number of regions must each be even (see Prop.~\ref{P:even}).

\begin{figure}[!ht]
\begin{center}
\includegraphics[width=2.5in]{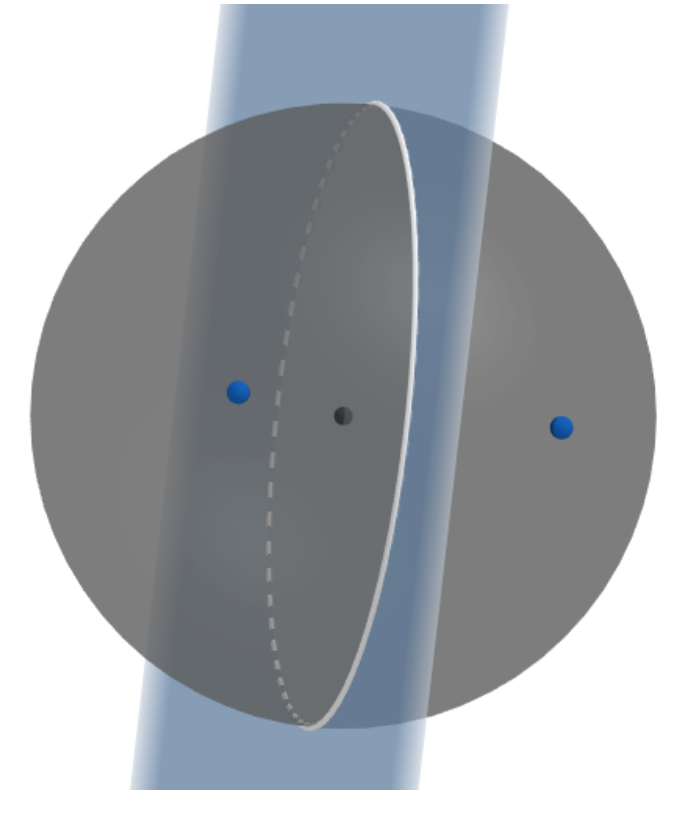}
\caption{\label{fig:2 pt sphere} Two points (in blue) and their bisecting great circle.}
\end{center}
\end{figure}

\subsection{From the plane to the sphere}\label{PlanetosphereSec}
Rather than derive results for the sphere from scratch, we first show how embedding a planar configuration onto a sphere affects the region count. This will allow us to translate several results from Section~\ref{SS:plane} to the sphere. We will use the graph $G_S$ generated by the perpendicular bisectors, but we do not need the bounding circle we introduced in Section~\ref{SS:plane}. The regions that touch that bounding circle will  now correspond to unbounded regions in the plane. 

We assume the configuration $S\subset \R^2$ has no parallel perpendicular bisectors. Then the construction we will use is easy to visualize. Begin with a $S\subset \R^2$ with centroid $A,$  then construct a large sphere centered at $A.$ Now project the points of $S$ onto the northern hemisphere. We write $\overline{S}$ for the image of $S$ on the sphere, and write $G_{\overline{S}}$ for the graph associated to the central hyperplane arrangement of bisecting great circles of $\overline{S}.$

\begin{thm}\label{T:sphere}
Let $S\subset \R^2$ be a configuration that does not generate any parallel bisectors. Suppose the associated graph $G_S$ contains $b$ bounded regions and $u$ unbounded regions. Let $\overline{S} \subset S^2$ and  $G_{\overline{S}}$ be as above. Then   $G_{\overline{S}}$ contains $u+2b$ distinct regions, so $\overline{S}$ generates $u+2b$ distinct orderings.

\end{thm}

\begin{proof}
First, note that the three bisectors generated by a spherical triangle  intersect at a common point. Moreover, points on the sphere that are concyclic (lie on a common intersection of a plane with the sphere) also behave as they do in the plane, generating bisectors that all intersect at the center of the circle. The distinction on the sphere is that these intersections occur twice, at antipodal points.

%Thus, given a configuration $S$ of $n$ points in the plane as described, place $n$ points on the sphere in a configuration with the same occurrences of concyclicity and coinciding bisectors as $S$. In particular, place these points within one hemisphere such that the bounded regions formed (and corresponding exactly to the bounded regions in the plane) are contained within that hemisphere.  

We may assume that the points of $\overline{S}$ all lie in the northern hemisphere with the equator as the bounding great circle. In this hemisphere, we will determine the number of regions whose great circle boundaries  do not touch the equator --- type A regions --- and the number of regions in which the equator is a bounding great circle --- type B regions. Then every bounded region of $G_S$ corresponds to a type A region, and every unbounded region of $G_S$ corresponds to a type B region.  As great circles intersect each other twice, all of the intersections from the planar arrangement will be mirrored on the antipodal hemisphere, doubling the number of bounded regions and closing the regions that were unbounded in the first hemisphere. This gives a total of $u+2b$ regions.

\end{proof}

Note that the equator plays the same role as the large circle we used in Section~\ref{SS:plane} in constructing the graph of perpendicular bisectors. 

\begin{ex} In Fig.~\ref{fig:sphere and plane}, we see an example of a translation of a planar configuration to a spherical configuration. In the plane, we have 5 points, 4 of which are concyclic such that we have 2 bisectors that coincide. Placing 5 points on the sphere with these same properties, we see that the 34 regions on the hemisphere shown correspond to the 34 regions in the plane. Since the opposite hemisphere will be its mirror image, there will be an additional 16 bounded regions there. Thus, this spherical configuration gives a total of 50 achievable orderings.

\begin{figure}[!ht]
\begin{center}
\includegraphics[width=5in]{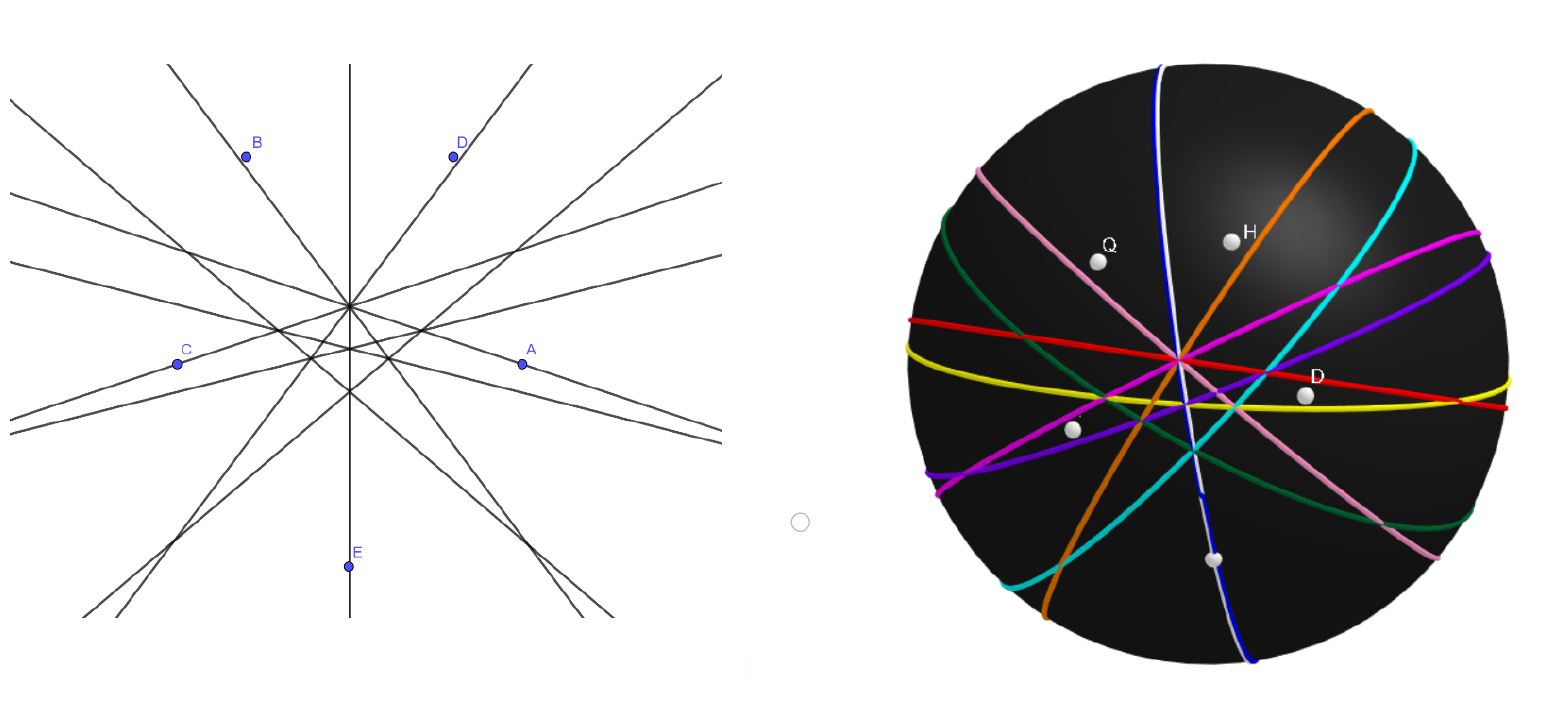}
\caption{\label{fig:sphere and plane} A configuration of 5 points (with 4 concyclic points and one pair of coinciding bisectors) in the plane, replicated on the sphere.}
\end{center}
\end{figure}

\end{ex}

This formula not only allows us to determine the number of orderings generated by particular spherical configurations, but also will make it possible to fill in gaps between the maximum and minimum on the sphere using only our data from the plane. In the next section, we apply this result to determine the maximum, minimum, and intermediate ranges on the sphere.

\subsection{Maximum, minimum, and gap filling}

The maximum number of orderings that can be obtained from a configuration of $n$ points on the sphere is given by another degree 4 polynomial, presented in the next theorem, originally found by Cover \cite{cov} in 1967.  These maximum values for $n \geq 2$ points appear as sequence A087645 in the OEIS \cite{s}. Our proof uses Proposition \ref{P:1max} and Theorem  \ref{T:sphere},  but we note that a  proof using Euler's polyhedra formula is straightforward.

\begin{thm}\label{SphereMax}
The maximum number of orderings that can be obtained from an arrangement of $n > 1$ points on the surface of the sphere is
$$\frac{1}{12}(3n^4-10n^3+9n^2-2n+24).$$
\end{thm}

\begin{proof}
Consider a set $S$ of $n$ points in the plane that generate the maximum number of achievable orderings. The $\binom{n}{2}$ distinct bisectors each generate two unbounded regions, meaning there are $2\binom{n}{2}$ unbounded regions and $M(n,2)-2\binom{n}{2}$ bounded regions. 

Now form the configuration $\overline{S} \subset S^2.$ Then, by Theorem \ref{T:sphere},  $\overline{S}$ contains  $u+2b$ regions, where $u$ and $b$ are the numbers of unbounded and bounded regions, respectively, in $S.$ Thus,

\begin{align}
    u+2b & = 2\binom{n}{2}+2\bigg[M(n,2) -2\binom{n}{2}\bigg] \nonumber \\
    & = n(n-1)+2\bigg[\frac{1}{24}\big(3n^4-10n^3+21n^2-14n+24\big)-n(n-1)\bigg]  \nonumber \\ &= \frac{1}{12}(3n^4-10n^3+9n^2-2n+24). \nonumber
\end{align}

\end{proof}

We now turn our attention to the minimum. The following lemma translates results about concyclic arrangements of points in the plane to  concyclic arrangements on the sphere.

\begin{lem}\label{L:circ}
An arrangement of concyclic points on the sphere gives the same number of orderings as that arrangement would in the plane.
\end{lem}

\begin{proof}
In the plane, $n$ concyclic points generate bisectors that intersect at the center of the circle. Such a configuration creates only unbounded regions in the plane, precisely twice  the number of distinct perpendicular bisectors created. On the sphere, the bisecting great circles formed by the $n$ points will still all intersect at the center  of that circle, creating two regions (from antipodal symmetry) for each distinct perpendicular bisector. Thus, a given concyclic arrangement gives the same number of regions on the sphere as in the plane.

\end{proof}

The next result establishes the minimum number of regions achievable on a sphere, achieved by a concyclic arrangement of $n$ points equally spaced on a circle.

\begin{thm}\label{T:minsphere}
Let $n>4.$ Then the minimum number of orderings that can be obtained from an arrangement of $n$ points on the sphere is $2n$, occurring precisely when the points are equally spaced concyclically on the sphere.
\end{thm}

\begin{proof}
Let $S\subset S^2$ be our collection of $n$ points on the sphere. First, by Lemma~\ref{L:circ}, we know that  $n$ points spaced equally on the equator of a sphere generate precisely $2n$ orders. It remains to show that this is the minimum possible, i.e., the  number of regions is greater for other spherical configurations. To see this, first suppose $S$ does not lie on a circle, but the bisecting great circles generated by $S$ produce at most $2n$ regions. Then the points of $S$ are not coplanar (considered as a subset of $\R^3$), so, by Theorem 1.1 of \cite{pps}, these points determine at least $2n-5$ distinct directions in $\R^3.$ Each such direction determines a unique normal vector direction for the perpendicular bisecting plane. All of these bisecting planes pass through the center of the sphere (because the points lie on that sphere), so the point configuration generates at least $2n-5$ distinct great circles on the sphere. This gives us at least $4n-10$ regions, since each new great circle adds a minimum of two new regions. 

Now $2n < 4n-10$ whenever $n>5,$ so, if $n>5,$ the points of $S$ must be concyclic on the sphere. It remains to show that they are also equally spaced on that circle. For a given bisecting great circle $C,$ let $\alpha(C)$ denote the number of distinct pairs of points of $S$  with bisecting great circle $C.$ It is clear that $1\leq \alpha(C) \leq \lfloor \frac{n}{2}\rfloor.$ 

Let $\mathcal{C}$ denote the collection of all the bisecting great circles. Then an incidence count shows 
$$\sum_{C \in \mathcal{C}}\alpha (C)={n \choose 2}$$ since every pair of points produces a great circle. Thus, $|\mathcal{C}|$ will be minimized when $\alpha(C)=\lfloor\frac{n}{2}\rfloor$ for all $C\in\mathcal{C},$ i.e., when $\alpha(C)$  is as large as possible. But this occurs only when the points of $S$ are equally spaced on a circle.

Finally, when $n=5$, the reader can check that configurations $S$ which are either non-concyclic or cyclic with non-evenly spaced points produce more than $10$ regions.

\end{proof}

When $n=4$, the minimum number of regions created is still $2n=8$. To create eight regions, the four points must be concyclic so that all of the great circles will intersect at the same pair of antipodal points. But the points need not be evenly spaced since rectangles also minimize the number of distinct great circles. This situation is completely analogous to the minimum value in the plane (see Thm.~\ref{T:min}), where the $n=4$ case is also exceptional.

For the remainder of this section, we focus on filling in the gaps between the maximum and minimum values on the sphere. The next proposition guarantees that odd numbers of orderings are not  achievable by any spherical configurations.
\begin{prop}\label{P:even}
Any arrangement of $n>1$ points on a sphere generates an even number of orderings.
\end{prop}
\begin{proof}
Given $n>1$ points and all of their pairwise bisecting great circles on the sphere, divide the sphere along one of these great circles so that it is separated into two hemispheres. Since both hemispheres are bounded by this great circle, each hemisphere contains only bounded regions. Then  antipodal symmetry implies that the total number of regions on this sphere (and so, the number of orderings generated) is twice the number of regions found on one of these hemispheres, and therefore is always even.

\end{proof}

We can use ideas similar to those given in the proof of Theorem~\ref{T:gap1}  to get a similar gap-filling result for the sphere. To accomplish this,  convert a given linear configuration to a concyclic one that lives on a sphere.  We leave the details of this argument to the interested reader.

\begin{prop}\label{P:gapsphere1}
With a configuration of $n > 3$ concyclic points on the sphere, $2k$ orderings are achievable for all $k$ such that $n \leq k \leq \binom{n}{2}$.
\end{prop}

As in the plane, it is probably not possible to determine the ranges of achievable orderings as a function of $n$ in general. We offer one last result in this direction, however, to add another interval of known achievable numbers of orderings on the sphere. 

\begin{prop}
With $n \geq 5$ points on a sphere, for every $t$ such that $(n-1) \leq t \leq \binom{n-1}{2}$, there is a spherical configuration that achieves $2nt$ orderings consisting  of $n-1$ concyclic points and 1 additional point not on that circle.
\end{prop}

\begin{proof}
Let $t$ be chosen with $(n-1) \leq t \leq \binom{n-1}{2}.$ By Lemma~\ref{L:circ}, we can find a circular arrangement of $n-1$  points which divide the surface of the sphere into $2t$ regions. Starting with one of these arrangements, the addition of an $n$th point that is not concyclic with the others generates $n-1$ new, distinct, bisecting great circles, none of which contain the center of the circle formed by the concyclic points (i.e., one of the two intersection points of all great circles generated strictly by the $n-1$ concyclic points). Then each of the $n-1$ new great circles passes through each of the $2t$ existing regions created by the bisecting great circles of the $n-1$ concyclic points. This means that each existing region will be divided into $n$ new regions,  yielding a total of $2nt$ regions.

\end{proof}

We remark that embedding other planar configurations that achieve intermediate values (in the manner of  Props.~\ref{P:cir} and \ref{P:par}, for instance)  will produce other gap-filling results on the sphere.

\subsection{Configurations with antipodal pairs}

Our main results treating spherical configurations thus far have been obtained by embedding planar configurations onto a hemisphere of a sphere. But this does not address all possible configurations. For example, the six vertices of a regular octahedron generate nine reflection planes, and these planes divide the octahedron into 48 triangular regions. The regions, which  correspond to the 48 elements of the symmetry group of the octahedron (the Coxeter group  $ S_4\times \mathbb{Z}_2$) are identical to the regions generated by the perpendicular bisectors for pairs of vertices.  Hence, this arrangement of six points generates 48 orderings of those points, with each ordering corresponding to one of the regions (also called \textit{cells}) in the hyperplane arrangement. See Fig.~\ref{F:oct} for the spherical regions determined by the octahedron, and see Example~\ref{E:plato} for an investigation of all the Platonic solids

\begin{figure}[htb] %  figure placement: here, top, bottom, or page
   \centering
\includegraphics[width=3in]{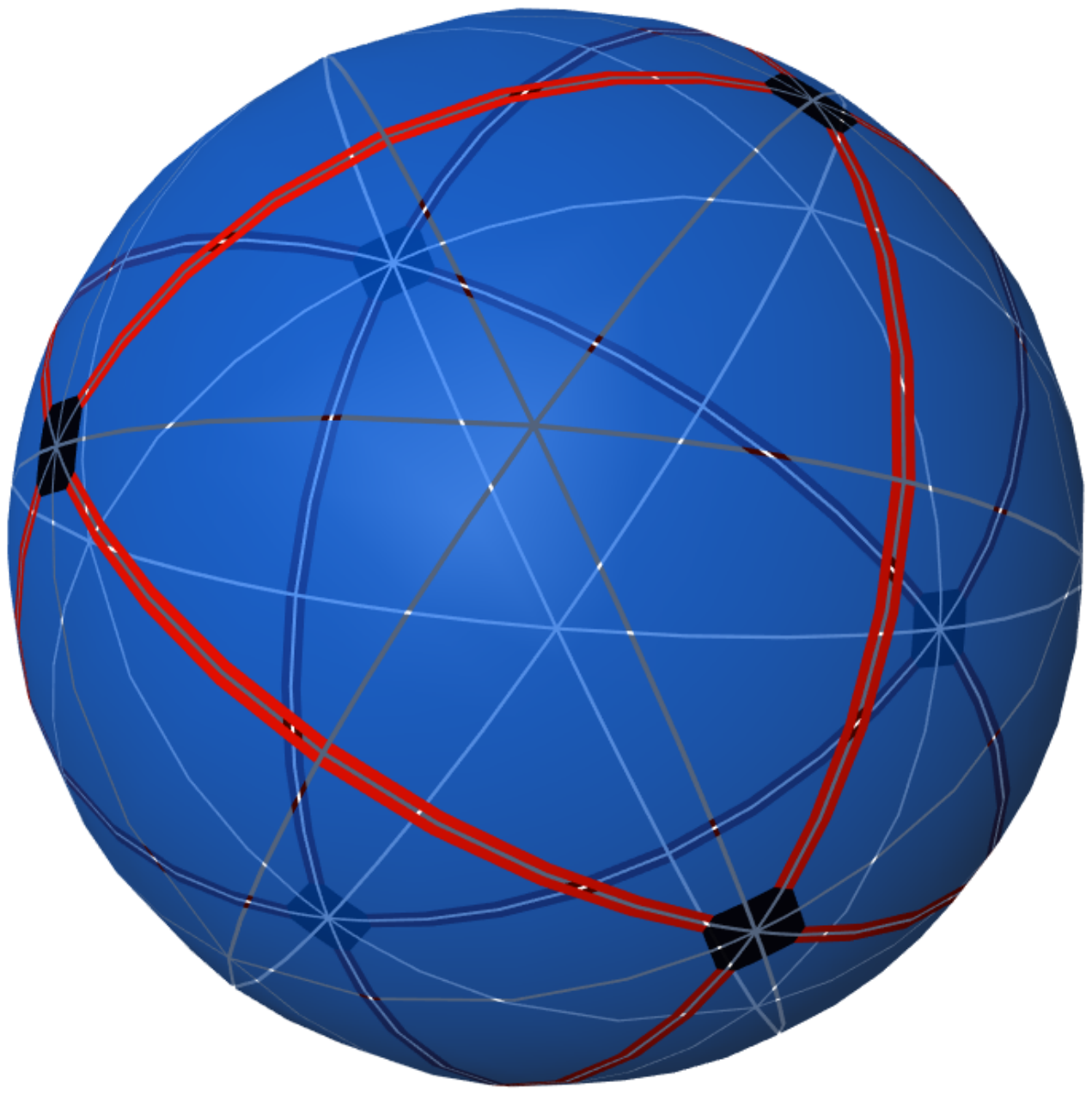} 
 \caption{A regular octahedron is an antiprism with a triangular base. The formula from Thm.~\ref{T:double} gives 48 regions, and each region corresponds to an element of the Coxeter group $B_3\cong S_4 \ltimes \mathbb{Z}_2.$ See Example \ref{E:plato} for more information about the  Platonic solids.}
\label{F:oct}
\end{figure}

%\begin{figure}[h] %  figure placement: here, top, bottom, or page
%   \centering
%\includegraphics[width=3in]{icosa} 
%
% \caption{Twelve points placed at the vertices of a regular icosahedron generate 15 great circle bisectors, yielding 120 regions. The edges of the icosahedron are in red; the  thin mirror lines are the bisecting great circles. Each of the twenty triangular faces of the icosahedron is divided into six triangular cells by the bisectors.}
%\label{F:ico}
%\end{figure}

To begin to scratch the surface of the possibilities involving antipodal pairs, we present one final result in this section. This result concerns ``doubled'' maximal configurations of the form $S\cup (-S),$ where $-S$ is the image of $S$ under the antipodal map (also called \textit{central inversion}) in $S^2.$ For points $P$ and $Q$ on the sphere, we write $c_{P,Q}$ for the perpendicular bisecting great circle determined by $P$ and $Q.$ We also use $P'$ and $Q'$ to denote the antipodes of $P$ and $Q,$ respectively We will need the following lemma, which follows immediately from the fact that the antipodal map fixes all bisecting great circles.

\begin{lem}\label{L:pairs}
Suppose $P$ and $Q$ are points on the surface of a sphere, and let $c_{P,Q}$ be the perpendicular bisecting great circle determined by $P$ and $Q.$ Then $c_{P,Q}=c_{P',Q'}.$
\end{lem}

\begin{thm} \label{T:double}
Let $S\subset S^2$ be a collection of $n$ points contained in some hemisphere of $S^2$ that produce the maximum number of regions on a sphere, and let $T=S\cup (-S)$ be the doubled configuration of $2n$ points (with $n \geq 3$). Then the number of regions formed is $$\frac{1}{3} \left(3 n^4-4 n^3+n+6\right).$$ 
\end{thm}

\begin{proof} To derive the formula, we will first count the number of vertices of degree 4, 6, and 8 (the only degrees possible for this configuration) in the (planar) graph formed by the intersections of all the perpendicular bisecting great circles on the sphere. We will then use Euler's formula to compute the number of regions formed by this central hyperplane arrangement.

Let $v_4, v_6,$ and $v_8$ be the number of vertices of degrees 4, 6, and 8. Then it immediately follows from Euler that the number of regions is $$R=2+v_4+2v_6+3v_8.$$ We will now compute $v_4, v_6,$ and $v_8.$ 

\begin{enumerate}
\item Degree 4 vertices: A vertex of degree 4 can only arise when we choose four distinct points from $T.$ We consider several ways this can happen.
\begin{enumerate}
\item All four points are in $S.$ Label the four points  $P_1, P_2, P_3, P_4$ and note that there are three distinct parings of the four points, each of which produces two antipodal vertices of degree 4: $c_{1,2} \cap c_{3,4}, c_{1,3}\cap c_{2,4},$ and $c_{1,4}\cap c_{2,3}.$ (We write $c_{i,j}$ instead of the more cumbersome $c_{P_i,P_j}$ throughout this proof.)

This produces a total of $6{n \choose 4}$ vertices of degree 4. Further, Lemma~\ref{L:pairs} ensures that the bisecting great circles produced from pairs of points in $-S$ coincide with the corresponding bisecting great circles produced by choosing pairs of points from $S.$  Thus, this case also accounts for choosing four points from $-S.$

\item Three points are in $S$ and one is in $-S.$ There are two cases to consider depending on whether the point chosen from $-S$ is the antipode of one of the points chosen from $S.$ (We also remark that subsets of four points where three are chosen from $-S$ and one is chosen from $S$ are also counted in this case, by Lemma~\ref{L:pairs}.)

\textbf{Case 1}: No two points from our set of four points are antipodal. Write $P_1,P_2,P_3,P_4'$ for the four points. There are ${n \choose 4}$ ways to select the four indices, and four ways to choose the index of the point in $-S.$ For a given subset of indices, this gives 12 pairings, each of which produces two antipodal vertices of degree 4, but these can be paired off using Lemma~\ref{L:pairs}. For instance, the pairing  $c_{1,2}\cap c_{3,4'}$ coincides with the pairing $c_{1,2}\cap c_{3',4}$. 

Then this case gives a total of $12{n \choose 4}$ vertices of degree 4.
\medskip

\textbf{Case 2}: The point chosen from $-S$ is antipodal to one of the points chosen from $S.$ This time, we write $P_1,P_2,P_3,P_1'$ for the four points, where $P_1'$ is antipodal to $P_1.$ There are ${n \choose 3}$ ways to choose the indices and three ways to pick the antipode, but several of these 12 pairings do not produce new vertices of degree 4. For example, the pairing of the form $c_{1,2}\cap c_{1',3}$ is identical to the pairing $c_{1,2}\cap c_{1,3'}$ (by Lemma~\ref{L:pairs}), which corresponds to a vertex of degree 6.

The only pairings that produce vertices of degree 4 are $c_{1,1'} \cap c_{2,3}, c_{2,2'} \cap c_{1,3}$, and $c_{3,3'} \cap c_{1,2}.$ We conclude that there are  $6{n \choose 3}$ vertices of degree 4 in this case.

%As in the first case, there are three pairings, each of which produces two antipodal vertices of degree 4. These three pairings are $c_{1,2} \cap c_{1',3}, c_{1,3}\cap c_{1',2},$ and $c_{1,1'}\cap c_{2,3}.$ 
%
%To count the number of such sets, we can choose the three indices in ${n \choose 3}$ ways and we can then choose the antipode in 3 ways. 
%There are a total of $12{n  \choose 3}$ vertices of degree 4 produced in this case. We list these pairings in Table~\ref{Ta:pairs}.
%
%\begin{table}[htp]
%\caption{Three points are chosen from $S$ and one from $S'.$ Pairings appearing on the diagonal of the table each produce two antipodal vertices of degree 4, but pairings below the diagonal produce the same pairs of antipodal vertices of degree 4 as the pairings above the diagonal.}
%\begin{center}
%\begin{tabular}{|c|c|c|c|} \hline
%Points & Pairing 1 & Pairing 2 & Pairing 3 \\ \hline
%$P_1,P_2,P_3,P_1'$ &  $c_{1,1'} \cap c_{2,3}$ & $c_{1,2} \cap c_{1',3}$ &$c_{1,3} \cap c_{1',2}$  \\ \hline
%$P_1,P_2,P_3,P_2'$ &  $c_{1,2} \cap c_{2',3}$ & $c_{2,2'} \cap c_{1,3}$ &$c_{1,2'} \cap c_{2,3}$  \\ \hline
%$P_1,P_2,P_3,P_3'$ &  $c_{1,3} \cap c_{2,3'}$ & $c_{2,3} \cap c_{1,3'}$ &$c_{3,3'} \cap c_{1,2}$  \\ \hline
%\end{tabular}
%\end{center}
%\label{Ta:pairs}
%\end{table}%
%

\item Two points are chosen from $S$ and two from $-S.$ This time, we separate into three cases.

\textbf{Case 1}: The two points from $-S$ are the antipodes of the two points from $S.$ We write $P_1,P_2,P_1',P_2'$ for the four points and note that these four points form a rectangle. This produces two antipodal vertices of degree 8 (counted below), but no vertices of degree 4. 

\textbf{Case 2}: One point from $-S$ is antipodal to one point in $S.$ There are ${n \choose 3}$ ways to choose the indices and three ways to select the point whose antipode appears. (Lemma~\ref{L:pairs} tells us the sets $\{P_1,P_1',P_2,P_3'\}$ and $\{P_1,P_1',P_2',P_3\}$ will produce identical bisecting great circles.) Then the count is completely analogous to case 2 above, where we considered sets of the form $\{P_1,P_2,P_3,P_1'\}$. This case yields another $6{n \choose 3}$ vertices of degree 4.

\textbf{Case 3}: No antipodal pairs are chosen. Then we write $P_1,P_2,P_3',P_4'$ for the four points. The reader can check that the only pairings that produce new vertices of degree 4 are the following:

$$c_{1,2'}\cap c_{3,4'} \hskip.5in c_{1,3'}\cap c_{2,4'} \hskip.5in c_{1,4'}\cap c_{2,3'} $$

All other pairings produced will coincide with pairings produced in the first case (where all four points were selected from $S$) or with one of these three pairings. Then this case produces $6{n\choose 4}$ vertices of degree 4.

\end{enumerate}
\item Vertices of degree 6: This time, we select three points from $S\cup (-S).$ There are two cases.

\textbf{Case 1}: The three points are in $S.$ Then there are ${n\choose 3}$ ways to select the points, and each selection produces two antipodal vertices of degree 6. 

\medskip

\textbf{Case 2}: Two of the points are in $S,$ and one is in $-S.$ In this case, we have ${n \choose 3}$ ways to select the points and three ways to choose the index corresponding to the point in $-S.$ This gives $6{n\choose 3}$ vertices of degree 6.

\item Vertices of degree 8: Vertices of degree 8 are only produced by collections of four points of the form $P,Q,P',Q',$ where $P,Q \in S.$ Each such set gives two antipodal vertices of degree 8.

\end{enumerate}

Putting all of this together gives 
\begin{eqnarray*}
v_4&=&24{n \choose 4}+12{n \choose 3}. \\
v_6&=&8{n\choose 3} \\
v_8&=& 2{n\choose 2}
\end{eqnarray*}

Then the number of regions is $2+v_4+2v_6+3v_8=\frac{1}{3} \left(3 n^4-4 n^3+n+6\right).$

\end{proof}

We can check the formula given in Theorem~\ref{T:double} with a regular octahedron. In this case, we set $S$ to be an equilateral triangle of suitable size so that $S\cup(-S)$ forms the vertices of a regular octahedron. Then, by the theorem, we can determine the number of regions generated by evaluating the above formula at $n=3.$ This gives a total of 48 regions, which agrees with the size of the Coxeter group associated with the octahedron. See Fig.~\ref{F:oct}. In fact, we can say a little more in the case when $|S|=3.$ By Theorem~\ref{T:double}, the number of regions of any doubled triangle is 48, so the octahedron formed by a scalene triangle living in one hemisphere of $S^2$ gives the same number of regions as the regular octahedron. 

\begin{ex}\label{E:plato}  Platonic solids. It is an  interesting exercise to determine the number of regions the bisecting great circles determine for point configurations  that correspond to the vertices of the Platonic solids. In general, the planes of reflection for  these solids divide the solids into cells that are in one-to-one correspondence with the elements of the corresponding Coxeter symmetry group. Each reflection plane for a solid will be a bisecting great circle for some pair of vertices of the solid. 

But the converse is not true, in general. For the tetrahedron and octahedron, every  bisecting great circle arises from a reflection plane for the solid. However, this is false for the other three Platonic solids. For the cube, icosahedron, and dodecahedron, choosing two antipodal points will generate a bisecting great circle that is not a reflection plane for the solid. These ``additional'' great circles will double the number of regions in these cases, compared with the number of elements of the corresponding Coxeter  group. 

In Fig.~\ref{F:plato}, for each solid,  we show how each face is decomposed into regions by the bisecting great circles. Fig.~\ref{F:plato2} shows the icosahedron and the dodecahedron embedded on spheres, with the mirror lines of symmetry. The region counts are given in Table~\ref{Ta:plato}. We also point out that there are precisely $n^2$ distinct perpendicular bisecting great circles for the doubled configurations of Theorem~\ref{T:double}. For the octahedron, these correspond to the nine reflection planes of symmetry. 

\begin{table}[htp]
\caption{For $n=4, 6, 8, 12,$ or 20 points on a sphere, we list the minimum and maximum number of regions possible, along with the number of regions the bisecting great circles of pairs of vertices determine for each Platonic solid. The number of regions is equal to the size of the corresponding symmetry group for the tetrahedron and the octahedron, and is twice the size of the symmetry group for the other Platonic solids.}
\begin{center}
\begin{tabular}{|c|c|lc|c|} \hline
$n$ & Min & Platonic solid & \#& Max \\ \hline
4 & 8 & Tetrahedron & 24 & 24 \\ \hline
6 & 12 & Octahedron &  48 &172 \\ \hline
8 & 16 & Cube &96& 646 \\ \hline
12 & 24 & Icosahedron &240& 3852 \\ \hline
20 & 40 &Dodecahedron & 240& 33632 \\ \hline
\end{tabular}
\end{center}
\label{Ta:plato}
\end{table}%

\begin{figure}[htb] %  figure placement: here, top, bottom, or page
   \centering
\includegraphics[width=1.5in]{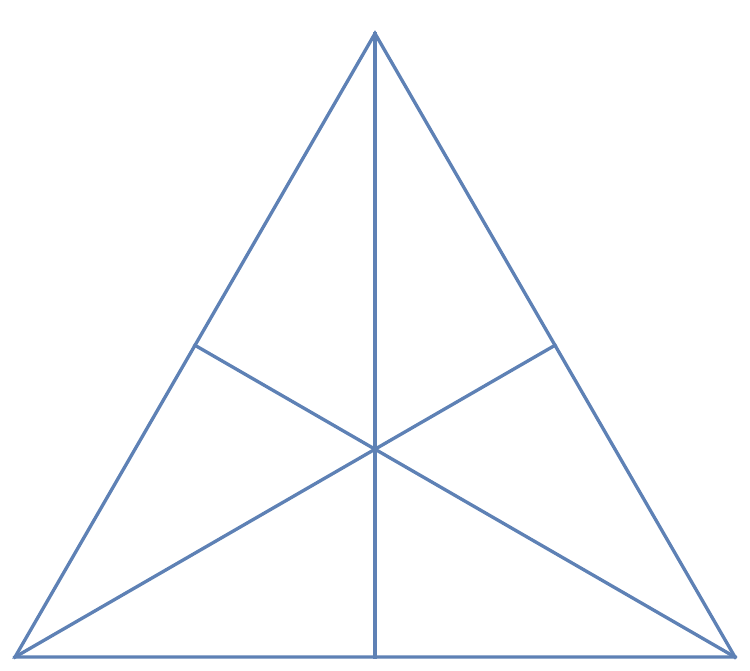} \hskip.1in
\includegraphics[width=1.5in]{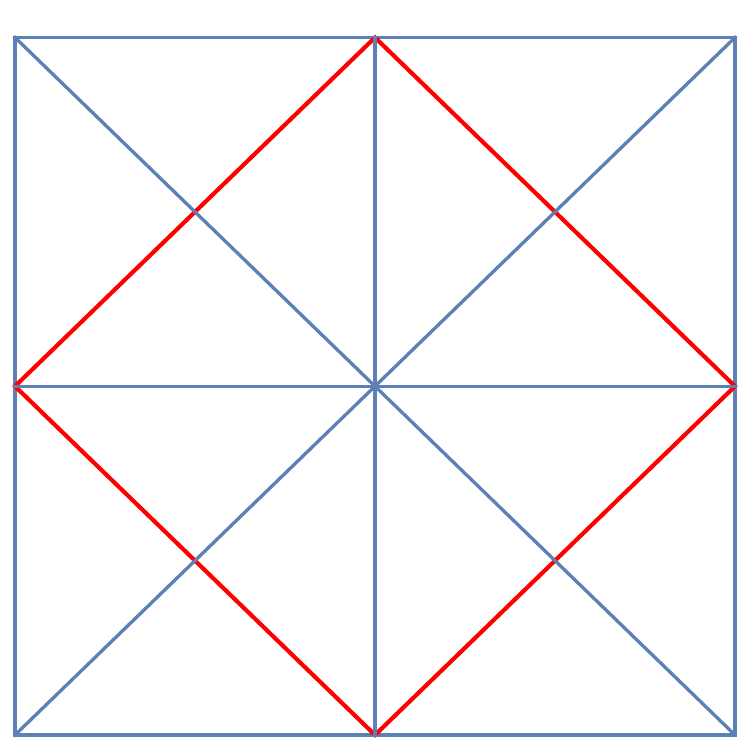} \hskip.1in
\includegraphics[width=1.5in]{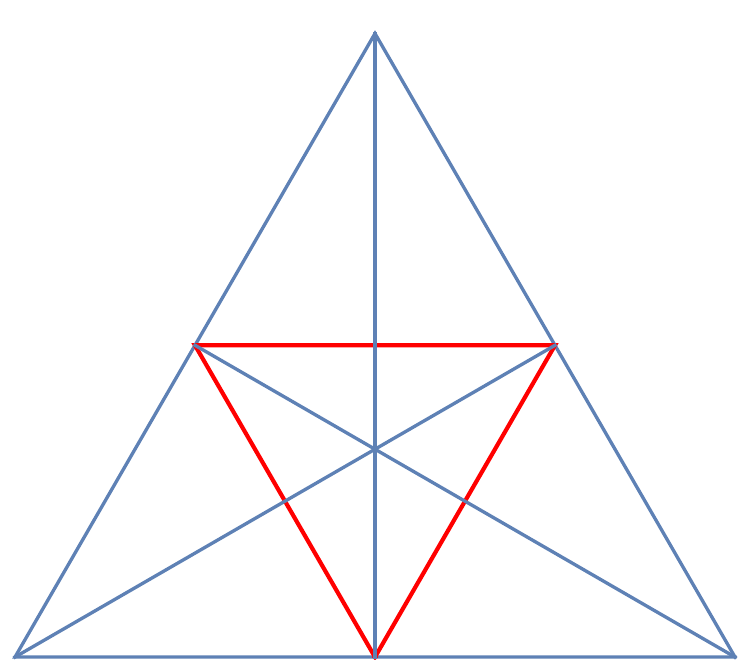} \hskip.1in
\includegraphics[width=1.5in]{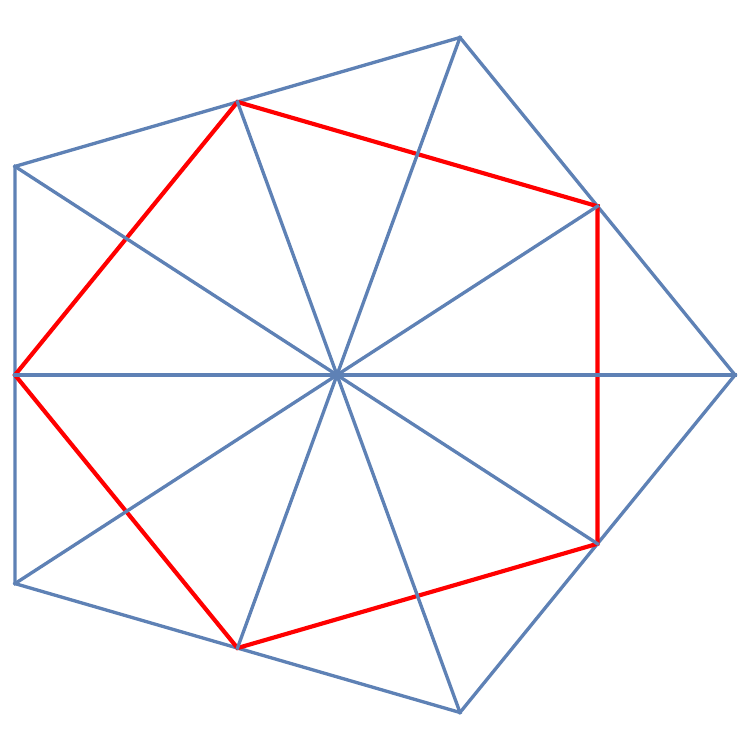} 

 \caption{From left to right: One face of the tetrahedron (or octahedron), the cube, the icosahedron, and the dodecahedron, divided into regions by the bisecting great circles. The red segments correspond to great circles determined by  pairs of antipodal points. For the cube, icosahedron, and dodecahedron, these do not  correspond to any reflections for that solid. Antipodal bisecting great circles do not appear for the tetrahedron. For the octahedron, bisecting great circles arising from pairs of antipodal vertices do correspond to reflections for the octahedron.}

\label{F:plato}
\end{figure}

\begin{figure}[htb] %  figure placement: here, top, bottom, or page
   \centering
\includegraphics[width=2.5in]{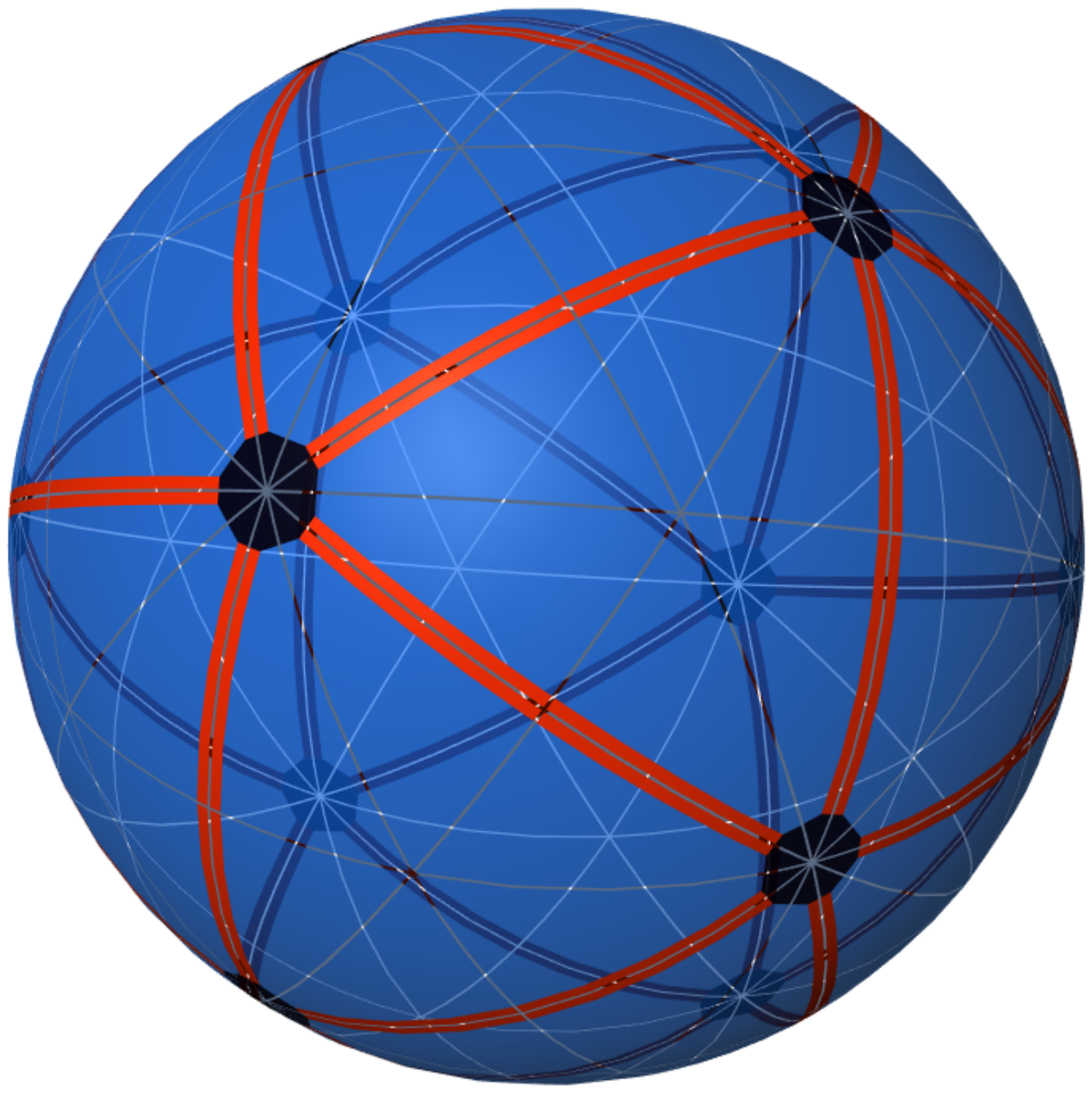} \hskip1in
\includegraphics[width=2.5in]{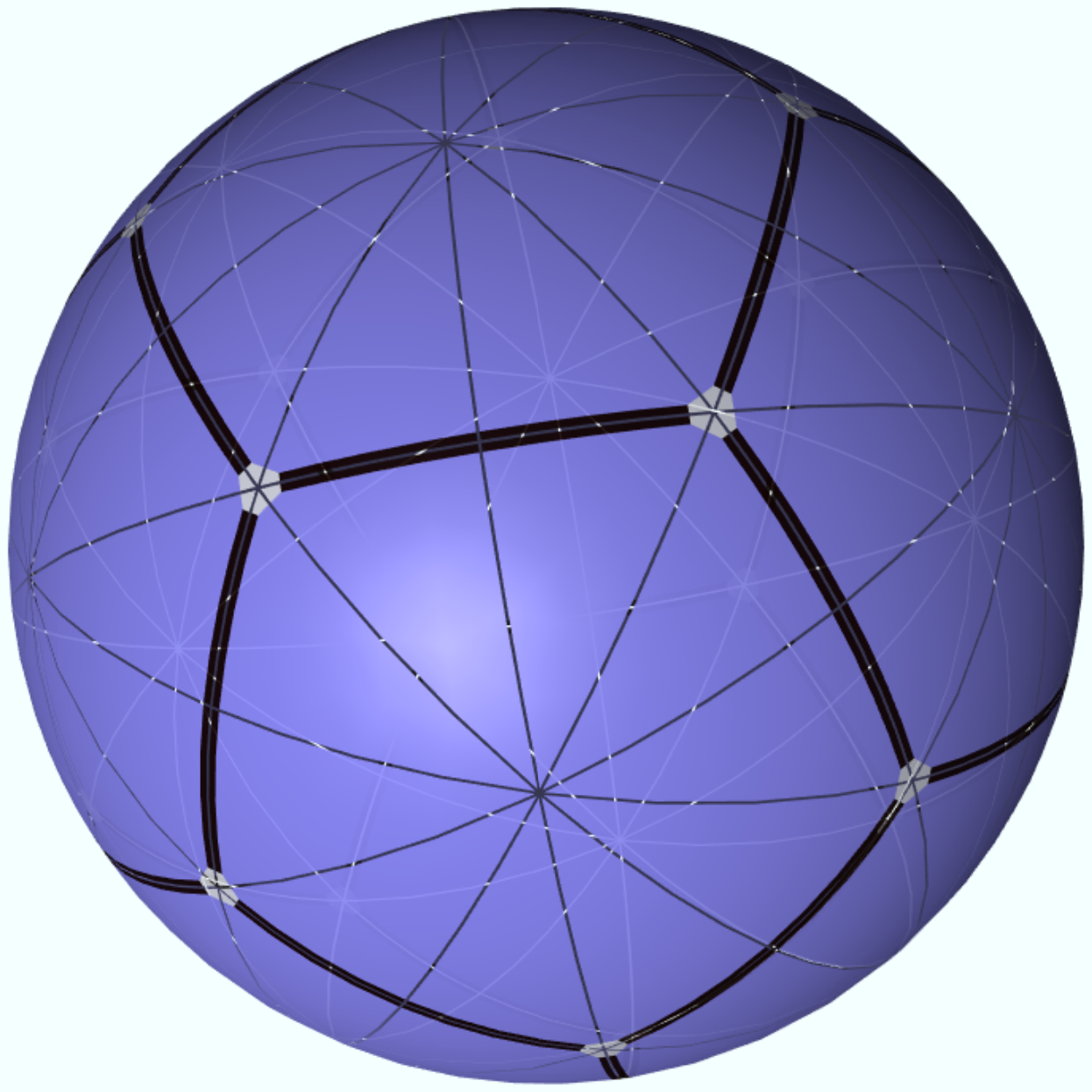} 
 \caption{The reflection planes for the icosahedron and dodecahedron coincide, and divide the sphere into 120 regions. The number of regions determined by the bisecting great circles of all pairs of points is 240.}

\label{F:plato2}
\end{figure}

\end{ex}

Finally, we remark that a doubled free point set on a sphere produces approximately one-fourth of the maximum number of regions $2n$ points can produce on a sphere. This follows from dividing the formula given in Theorem~\ref{T:double} by the maximum (for $2n$ points) given in Theorem~\ref{SphereMax}. 

\section{Generalizations}\label{S:gen}
Recall that voter preference lists were the original motivation for the discrete geometry problems considered here. In this section, we explore two generalizations: an application to \textit{weighted preferences} where a voter assigns real numbers as weights to issues to reflect their relative importance to that voter, and a version of the planar point problem where the ordering is determined by the average distance from a \textit{pair} of vantage points. We begin with the weighted version of our problem, and remark that all of our configurations will be linear or  planar in this section.

\subsection{Weighted preference lists}\label{SS:weight}
If a voter cares more about one issue than  another, it is easy to modify our approach to produce a preference list, as before. For example, if the voter cares twice as much about the issue represented on the $x$-axis than the issue represented on the $y$-axis, then a hypothetical voter situated at the origin will prefer a candidate at $(1,2)$ over a candidate at $(2,1),$ for instance. 

We are interested in how the weighted preferences determine an ordering of the candidates associated with our vantage point. A straightforward approach to this problem is to simply replace each  point $(x_k,y_k)$ in our set $S$ with $(w_xx_k,w_yy_k),$ where $w_x$ and $w_y$ are positive reals corresponding to the relative weights of the two issues. This dilatation will transform the arrangement of perpendicular bisectors, but will not change the number of regions. See Fig.~\ref{F:stretch}. In Prop.~\ref{P:weight1}, we show that this operation has the effect of replacing the standard Euclidean distance formula with a weighted version: $$d_w((a,b),(x_k,y_k))=\sqrt{w_{x}^2(x_{k}-a)^2+w_{y}^2(y_{k}-b)^2}. $$

\begin{figure}[htb] %  figure placement: here, top, bottom, or page
   \centering
\includegraphics[width=2.5in]{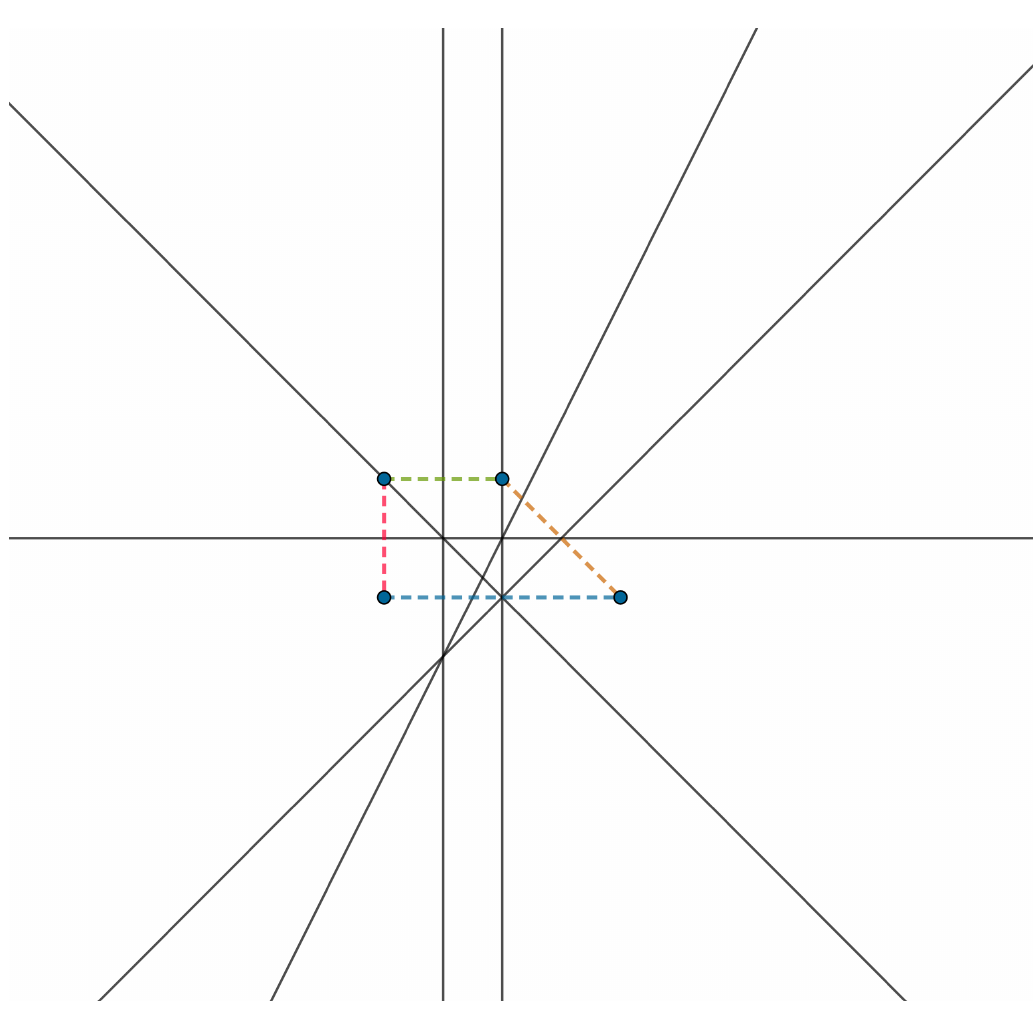} \hskip1in
 \includegraphics[width=2.5in]{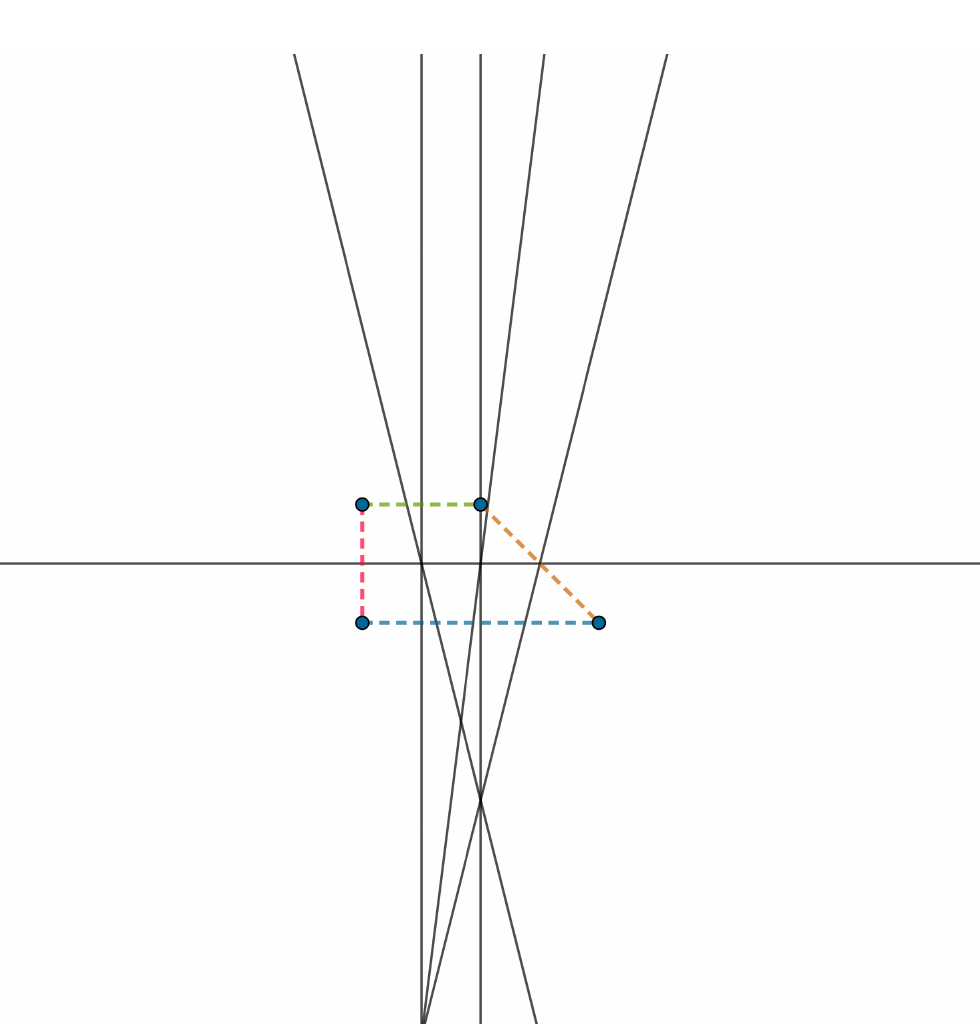} 
 \caption{Left: Unweighted configuration. Right: The same configuration, but with $w_x=2$ and $w_y=1.$ The equation for the lines in the weighted picture on the right are given in Prop.~\ref{P:weight2}. }
\label{F:stretch}
\end{figure}

\begin{prop}\label{P:weight1}
Let $P_1=(x_1,y_1), P_2=(x_2,y_2),$ and $V=(a,b)$ be three points in the plane, and assume the respective axes have weights $w_x$ and $w_y.$ Let $d_w((a,b),(x_k,y_k))=\sqrt{w_{x}^2(x_{k}-a)^2+w_{y}^2(y_{k}-b)^2} $ and write $\overline{P_k}=(w_xx_k,w_yy_k).$ Then $d_w(V,P_1)<d_w(V,P_2)$ if and only if $d(V,\overline{P_1})<d(V,\overline{P_2}).$ 
\end{prop}

\begin{proof}
Suppose $(a,b)$ is equidistant between $(x_{1}, y_{1})$ and $(x_{2}, y_{2})$ using weighted distance. Then we know 
$$\sqrt{w_{x}^2(x_{1}-a)^2+w_{y}^2(y_{1}-b)^2} = \sqrt{w_{x}^2(x_{2}-a)^2+w_{y}^2(y_{2}-b)^2}. $$
From this equation, we can see that
$$ \sqrt{(w_{x}x_{1}-w_{x}a)^2+(w_{y}y_{1}-w_{y}b)^2}=\sqrt{(w_{x}x_{2}-w_{x}a)^2+(w_{y}y_{2}-w_{y}b)^2},$$ 
which tells us that $(a,b)$ is equidistant from the points $(w_{x}x_{1}, w_{y}y_{1})$ and $(w_{x}x_{2}, w_{y}y_{2})$ using the standard Euclidean distance. 

\end{proof}

\begin{cor}\label{C:weight}
Let $S\subset\R^2$ be a finite set and let $V\in\R^2$ be a vantage point. If the axes have weights $w_x$ and $w_y,$ then we can determine the preference order for $V$ by replacing each $(x_k,y_k)\in S$ with $(w_xx_k,w_yy_k)$ and counting the regions determined using the standard, unweighted distance formula.
\end{cor}

We also note that a similar approach will work in higher dimensions. If a voter assigns weights to $d\geq 2$ different issues and each candidate corresponds to a point $P=(x_1,x_2,\dots,x_d)$ in $\R^d,$ then incorporating the weights simply transforms $P$ to $(w_1x_1,w_2x_2,\dots,w_dx_d).$ Then apply the Euclidean distance formula in $\R^d$ to the transformed coordinates to generate an ordering of the candidates as in the corollary.

By Prop.~\ref{P:weight1}, the procedure given in Cor.~\ref{C:weight} is equivalent to keeping the point set $S$ (and $V$) unchanged, but using the distance function $d_w.$ This is the point of our last result concerning weighted preferences. 

\begin{prop}\label{P:weight2}
Let $S=\{P_1,P_2,\dots,P_n\},$ where $P_i=(x_i,y_i),$ be a collection of $n$ points in the plane, and let $w_{x}$ be the  weight given to the $x$-axis, and $w_{y}$ be the  weight given to the $y$-axis. For each pair of points $P_i$ and $P_j,$ define a line  $l(i,j)$ as follows: 
$$w_x^2(x_{j}-x_{i}) x+w_y^2 (y_{j}-y_{i})y= \frac12\left({w_{x}^2}(x_{j}^2-x_{i}^2)+w_y^2(y_{j}^2-y_{i}^2)\right).$$
 Then the regions created by these lines determine the weighted preference order of the vantage point.
\end{prop}

\begin{proof}
This follows by setting the weighted distance between the point $(x,y)$ and $P_i$ equal to the weighted distance between $(x,y)$ and $P_j$ (as in the proof of \ref{P:weight1}), then simplifying. We omit the algebraic details.

\end{proof}

The lines separating the plane in Prop.~\ref{P:weight2} are not perpendicular bisectors. In particular, while the midpoint $(\frac{x_i+x_j}{2},\frac{y_i+y_j}{2})$ is on our line,  the slope of this line is $\displaystyle{-\frac{w_y^2(y_j-y_i)}{w_x^2((x_j-x_i)}},$ and so is not perpendicular to the line joining $(x_i,y_i)$ and $(x_j,y_j).$

\subsection{Two vantage points}\label{SS:2points}
Our final generalization concerns increasing the number of vantage points from one to two. As before, we are given $n$ points in $\R^d,$ but we now have vantage points $V_1$ and $V_2.$ As motivation, suppose that  two people with non-identical views wish to construct a single ordered list they can both agree on. One way to do this is to measure the average distance from $V_1$ and $V_2$ to each of the points in $S,$ then order the points from closest to farthest, using the average distance. (Although our ordering is determined by the \textit{average} distance from the two vantage points, we will use the \textit{sum} of those distances throughout the remainder of this section. These two approaches are obviously equivalent.)

Table~\ref{Ta:2pt} gives the results of a computer search for the minimum and maximum values for the number of orderings produced by moving two vantage points around the plane when $S\subset \R^2.$ (We point out that it took approximately 12 computers around 48 hours running in parallel to find 680 distinct orderings for six-point configurations.)

\begin{table}[htp]
\caption{Two vantage points for planar configurations: The results of a computer search for the minimum and maximum values for the number of regions produced using two moving vantage points for $S\subset \R^2.$ Note that the maximum is equal to $n!$ for $n\leq 5.$ The minimum values correspond to points equally spaced on a line --- see the values for $b_n$ in Table~\ref{Ta:2vpline}.}
\begin{center}
\begin{tabular}{c||ccccccccc}
$n$ &  2 & 3 & 4 & 5 & 6 & 7 & 8 & 9 & 10 \\ \hline \hline
Min &  2 &4&8&16&30&54&94&160&268 \\ \hline
Max &  2 & 6 & 24 & 120 & $\geq 680$ & ? & ? & ? & ? \\
\end{tabular}
\end{center}
\label{Ta:2pt}
\end{table}%

\subsubsection{The 1-dimensional case.}\label{SS:1dim}  In this subsection, we are given a collection of points $S \subset \R,$ along with  two vantage points $V_1,V_2 \in \R.$ We write  $S=\{P_1,P_2,\dots,P_n\}\subset \R$ and we assume the points are listed  in increasing order, so $P_1<P_2<\cdots <P_n,$ and our two vantage points are ordered so that $V_1<V_2.$

We begin with a complete description of the possible orderings that can be generated  in dimension 1. We will show that allowing a second vantage point does not increase the number of possible orderings, provided we ignore all orderings that produce ties. To complete this argument, we will need to understand how ties can be produced. Ties can happen in two distinct ways, described in the  next lemma.

\begin{lem}\label{L:ties}
Suppose the four points $P_i,P_j,V_1,V_2$ are collinear, placed on a number line, and $P_i$ and $P_j$ are tied in the ordering produced by average distance from two vantage points $V_1$ and $V_2.$  Assume $P_i<P_j$ and $V_1<V_2.$ Then either
\begin{enumerate}
\item $V_1<P_i<P_j<V_2,$ or
\item $P_i<V_1<V_2<P_j$ and  the midpoints of $\overline{V_1V_2}$ and $\overline{P_1P_2}$ coincide, i.e., $V_1+V_2=P_i+P_j.$
\end{enumerate}
\end{lem}
\begin{proof}
First, note that if  $V_1<P_i<V_2,$ then $d(V_1,P_i)+d(V_2,P_i)=d(V_1,V_2). $ Thus, if condition 1 is satisfied, then $P_i$ and $P_j$ will be tied. If condition 2 is satisfied, then  $d(V_1,P_i)+d(V_2,P_i)=V_1+V_2-2P_i$ and  $d(V_1,P_j)+d(V_2,P_j)=2P_j-V_1-V_2.$ But $$V_1+V_2-2P_i=2P_j-V_1-V_2 \text{ if and only if } V_1+V_2=P_i+P_j,$$ i.e., when the points are ordered $P_i<V_1<V_2<P_j,$ a tie will be produced precisely when  the midpoints of $\overline{V_1V_2}$ and $\overline{P_1P_2}$ coincide.

It remains to show that configurations not satisfying 1 or 2 do not produce ties. By the above argument, if the points are ordered $P_i<V_1<V_2<P_j$ (this is the ordering in condition 2), ties are only produced when the midpoints coincide. There are two potential orders of the four points $V_1, V_2, P_i,$ and $P_j$ that we consider, up to symmetry.
\begin{enumerate}
\item [a.] Suppose $V_1<P_i<V_2<P_j.$ Then $d(V_1,P_i)+d(V_2,P_i)=d(V_1,V_2)<d(V_1,P_j)<d(V_1,P_j)+d(V_2,P_j),$ so $P_i$ will precede $P_j$ in the ordering generated. (The case $P_i<V_1<P_j<V_2$ is handled by a symmetric argument.)
\item [b.] Suppose $V_1<V_2<P_i<P_j.$ Then $P_i$ will again precede $P_j,$ so no ties will be produced. (The argument for the case $P_i<P_j<V_1<V_2$ is symmetric.)
\end{enumerate}

\end{proof}

By placing the two vantage points close together, it is clear that any ordering of the points that is achievable with one vantage point on the line is also achievable in the 2-vantage point case. If we avoid configurations with ties, the converse is true. This is the point of the next theorem, which is the main result in this section.

\begin{thm}\label{T:2vpline}
Suppose $V_1, V_2, P_1,P_2, \dots,P_n \in \R$ with $P_1<P_2< \cdots <P_n,$ with vantage points $V_1<V_2,$ and let $\sigma$ be the ordering of these points generated by average distance from the two vantage points. We assume there are no ties in $\sigma.$ Then the ordering $\sigma$ can also be achieved using a single vantage point.
\end{thm}
\begin{proof}
Since there are no ties,  Lemma~\ref{L:ties} implies that the interval $[V_1,V_2]$ contains at most one point $P_i$ from our set. We set $V=\frac12(V_1+V_2),$ i.e., the midpoint of the segment $\overline{V_1V_2}.$ We will show that the ordering produced by the single vantage point $V$ is identical to the order produced by two vantage points $V_1$ and $V_2.$ We consider two cases.
\begin{itemize}
\item No $P_i$ satisfies $V_1<P_i<V_2.$ Then it is straightforward to show $d(V,P)=\frac12(d(V_1,P)+d(V_2,P))$ for all points $P$ in our set. This immediately gives us identical orders.
\item There is a unique index $k$ such that $V_1<P_k<V_2.$ By the argument given in the first case, we know $d(V,P_i)=\frac12(d(V_1,P_i)+d(V_2,P_i))$ for all points $P_i$ with $i \neq k .$ Further, the point $P_k$ will be ranked first in both the single vantage point  and the 2-vantage point cases. To see this, first note that if $i\neq k,$ then $d(V,P_k)<\frac12 d(V_1,V_2)< d(V,P_i).$ This tells us that $P_k$ will be listed first in the order produced using the single vantage point $V.$ 

But $P_k$ will also be listed first using our two vantage points since $d(V_1,P_k)+d(V_2,P_k)=d(V_1,V_2),$ but $d(V_1,P_i)+d(V_2,P_i)>d(V_1,V_2)$ for all $i\neq k.$ So, again, the two orders will be identical.
\end{itemize}

\end{proof}

The following corollary follows immediately from Theorems~\ref{T:gap1} and \ref{T:2vpline}.

\begin{cor}
Let $k$ be an integer with $2n-2\leq k \leq \frac12(n^2-n+2).$ Then there is a configuration $S\subset \R$ such that the number of distinct orderings  (with no ties)  produced by two vantage points on the line is $k.$  
\end{cor}

\subsubsection{Planar configurations}\label{SS:2ptplane}

When $S\subset \R^2,$  it becomes much more difficult to determine maximum and minimum values for the number of different orderings produced as the two vantage points move about the plane. If we draw perpendicular bisectors as in the single vantage point case, it is clear that an ordering of $S$ produced by the single vantage point $V$ can also be achieved in the 2-vantage point problem by placing the  two vantage points $V_1$ and $V_2$ in the same region that $V$ occupies. But, unlike the situation when all the points are collinear (including the vantage points) as in Section~\ref{SS:1dim},  we can achieve more orderings with two vantage points than we could with one when all the points are free to move about the plane.

In this section, we restrict to the case where the points of $S$ lie on a line in $\R^2.$ Even with this restriction, it seems difficult to find the maximum and minimum values for the number of orderings generated. See Table~\ref{Ta:2vpline} for the number of orderings produced with two vantage points when the points of $S$ are collinear. In this case, we further distinguish two cases: configurations with the points equally spaced on a line, and configurations where the spacing of the collinear points is unrestricted.

\begin{table}[htp]
\caption{Two vantage points in the plane for collinear configurations: $a_n$ is  the  number of orderings produced when the $n$  points are collinear; $b_n$ is  the number of orderings when the points  are collinear and equally spaced. In general, $a_n \leq 2^{n-1}$ and $b_n\leq 2F_{n+2}-2n,$ where $F_k$ is the $k^{th}$ Fibonacci number. See Propositions~\ref{P:2^n-1} and \ref{P:velo}.}
\begin{center}
\begin{tabular}{c||c|c|c|c|c|c|c|c|c|c} 
$n$ & 1 & 2 & 3 & 4 & 5 & 6 & 7 & 8 & 9 & 10 \\ \hline \hline
$a_n$ & 1 & 2 & 4 & 8 & 16 & 32 & 63 or 64 & ? & ? & ? \\ \hline
$b_n$ & 1 & 2 & 4 & 8 & 16 & 30 & 54 & 94 & 160 & 268 
\end{tabular}
\end{center}
\label{Ta:2vpline}
\end{table}

\begin{figure}[htb] %  figure placement: here, top, bottom, or page
   \centering
\includegraphics[width=2.5in]{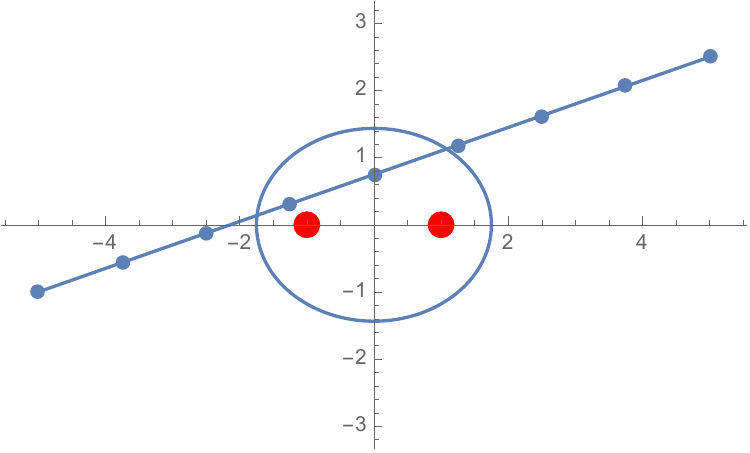} 
 \caption{The vantage points are $(-1,0)$ and $(1,0),$ and the points of $S$ are equally spaced on the line segment. The order is generated by expanding ellipses, with foci at the vantage points, shown in red.}
\label{F:2pts}
\end{figure}

When $S\subset  \R^2,$ we can visualize the  orderings produced when we use the average distance to the two vantage points by drawing a series of expanding ellipses, each with foci fixed at the two vantage points. To get an order for $S,$ simply record the order that the points of $S$ are hit as the ellipses expand.  See Fig.~\ref{F:2pts} for an example.

\begin{prop}\label{P:2^n-1}
Let $S$ be a collection of $n$ points on a line in $\R^2.$ Then the number of orderings produced using average distance to two vantage points in the plane is at most $2^{n-1}.$

\end{prop}

\begin{proof}
Fix two vantage points $V_1$ and $V_2$ in the plane, and suppose $P_i$ precedes $P_j$ in the ordering generated, where $1\leq i < j \leq n.$ (The other case is handled similarly.)  Suppose $k$ is between $i$ and $j,$ so $i<k<j.$ We will show that $P_k$ precedes $P_j$ in the ordering. This will ensure that the permutation of the points of $S$ generated will have the property that, for all $1 \leq k \leq n,$ the point $P_k$ can only appear  after either $P_{k-1}$ or $P_{k+1}$ appears (unless $P_k$ appears first). Then an elementary combinatorial argument shows that the number of  sequences satisfying this condition is $2^{n-1}.$ 

To see this, choose an ellipse with foci $V_1$ and $V_2$ where $P_j$ is on the boundary of the ellipse. Then $P_i$ must be in the interior of the ellipse (since $P_i$ precedes $P_j$ in the ordering). Thus, the line segment $\overline{P_i,P_j}$ is entirely contained in the convex hull  of the ellipse, since the segment and hull are both convex sets. Then $P_k$ will also be in the interior of the ellipse for all $k$ satisfying $i<k<j,$ so $P_k$ will precede $P_j$ in our ordering.

\end{proof}

We now treat the case where the points are equally spaced on the segment. We  further simplify our procedure by fixing the two vantage points and allowing the points of $S$ to move using planar isometries and dilations. This will be our approach throughout the remainder of this section. This reverses our usual procedure of fixing $S$ and moving $V_1$ and $V_2,$ but it is easy to show that these two approaches are equivalent.

Now fix the vantage points  at $(-1,0)$ and $(1,0),$ then choose two points $P_1=(x_1,y_1)$ and $P_n=(x_n,y_n)$ in the plane --- these will be the endpoints of our line segment. Then, if the points are equally spaced, we have  $$P_k=\left(\frac{(n-k-1)x_1+k x_n}{n-1},\frac{(n-k-1)y_1+k y_n}{n-1}\right).$$ Note that this approach depends on the values of five parameters: the $x$ and $y$ coordinates of the two endpoints, and the number of points $n.$ Different orderings will be produced as we vary the two endpoints of the line segment. This is consistent with our standard approach, where the five parameters are the coordinates of the two vantage points, in addition to $n.$

We let $c_n$ be the number of binary sequences of length $n$ that have the property that consecutive 0's and consecutive 1's cannot \textit{both} appear except at the beginning or the end of the sequence. For instance, the binary sequence 11011101000 is good, but 10001100 is bad. This integer sequence can be obtained from the integer sequence A000126 in \cite{s} by doubling every term in A000126, yielding the closed form formula for $c_n=2(F_{n+2}-n),$ where $F_k$ is the $k^{th}$ Fibonacci number. The next proposition shows that these binary sequences provide an upper bound for the number of orderings for equally spaced points on a line for the 2-vantage point problem.

\begin{prop}\label{P:velo} Let $b_n$ be the number of orderings possible with two vantage points in the plane, where the points are collinear, with points equally spaced. Then $b_n \leq c_{n-1}.$

\end{prop}
\begin{proof}

% \textbf{Part (1)}. Suppose $\sigma=s_1s_2\dots, s_n$ is a permutation of $\{1,2,\dots, n\}$ arising from a linear point set. Let $f(k)=\max_{i\leq k}s_i$ and $g(k)=\min_{i \leq k}s_i.$ For example, if $\sigma=3452671,$ then $s_1=3, s_2=4,$ and so on. Then $f(4)=5$ and $g(4)=2$ because 5 and 2 are the largest and smallest numbers appearing in the initial block 3452 of length 4 of $\sigma.$
%
%\textbf{Claim}: If $g(k) < j < f(k),$ then $j$ appears in $s_1,s_2,\dots, s_k$, the first $k$ elements of $\sigma.$  
%
%To see this, note that the line segment $l$ with endpoints $g(k)$ and $f(k)$ is completely contained in an ellipse, as in Fig. 1. Since the interior of an ellipse is convex, it must contain all of the points on $l.$ In fact, you can easily show that $s_k=f(k)$ or $g(k).$)  It is now a combinatorial exercise to show that the number of such sequences is $2^{n-1}.$

The proof uses  calculus. First, given a permutation of length $n,$ we create a binary sequence of 0's and 1's of length $n-1$ by recording the up-down sequence. For example, given the permutation 546732891, we get 01100110 (where 0 records a decrease  and 1 records an increase). %Thus, every permutation gives rise to such a sequence.

Now an ellipse with foci located at $(-1,0)$ and $(1,0)$ has equation $\displaystyle{\frac{x^2}{t^2}+\frac{y^2}{t^2-1}=1,}$ where the parameter $t$ corresponds to the positive $x$-intercept of the ellipse. We assume the line containing $S$ has slope $m$ and intercept $b,$ where $b,m>0.$  (We can reflect over the $x$-axis to get $b>0,$ if needed. If $m<0,$ then we modify the argument given below, swapping the 0's and 1's.)

We will show that the up-down binary sequence has the property that consecutive 0's can never occur except (possibly) at the beginning or the end of the sequence. First, we find the intersection of the ellipse and the line $y=mx+b.$ Here are the $x$-values of the two intersection points:

\begin{eqnarray*}
x_1&=&\frac{-\sqrt{4 b^2 m^2-4 \left(b^2-t^2+1\right) \left(m^2-\frac{1}{t^2}+1\right)}-2 bm}{2 \left(m^2-\frac{1}{t^2}+1\right)}\\
x_2&=&\frac{\sqrt{4 b^2 m^2-4 \left(b^2-t^2+1\right) \left(m^2-\frac{1}{t^2}+1\right)}-2 bm}{2 \left(m^2-\frac{1}{t^2}+1\right)}
\end{eqnarray*}
To determine whether we can get consecutive 0's in an associated permutation, we compute the derivatives $\frac{dx}{dt}$ at each of the intersection points. Note that $\frac{dx}{dt}>0$ when $x=x_2$ and $\frac{dx}{dt}<0$ at $x=x_1.$ (See Figure~\ref{F:velo}.)  Adding these derivatives gives $$\frac{dx}{dt}(x_1)+\frac{dx}{dt}(x_2)=\frac{4 b m t}{\left(\left(m^2+1\right) t^2-1\right)^2}.$$ But this is positive for $m,b>0.$ This implies that the vertical line $x=x_2$ is moving to the right faster than the vertical line $x=x_1$ is moving to the left. Thus, it is not possible for the sequence to have consecutive 0's in the interior of the associated permutation.

Finally, if $b=0,$ then $\frac{dx}{dt}(x_1)+\frac{dx}{dt}(x_2)=0,$ so the vertical lines are moving at the same speed. In this case, the above argument remains valid; in fact, in this case, consecutive 0's \textit{and} consecutive 1's can only occur at the beginning or end of the sequence.

Finally, if $b>0$ and $m<0,$  then the same argument will produce the same conclusion, where $0's$ and $1's$ are swapped.

\end{proof}

In Fig.~\ref{F:velo}, we show the vertical lines through the intersection points of the line $y=mx+b$ and the expanding ellipses with foci $(-1,0)$ and $(1,0).$ 

\begin{figure}[htb] %  figure placement: here, top, bottom, or page
   \centering
\includegraphics[width=3in]{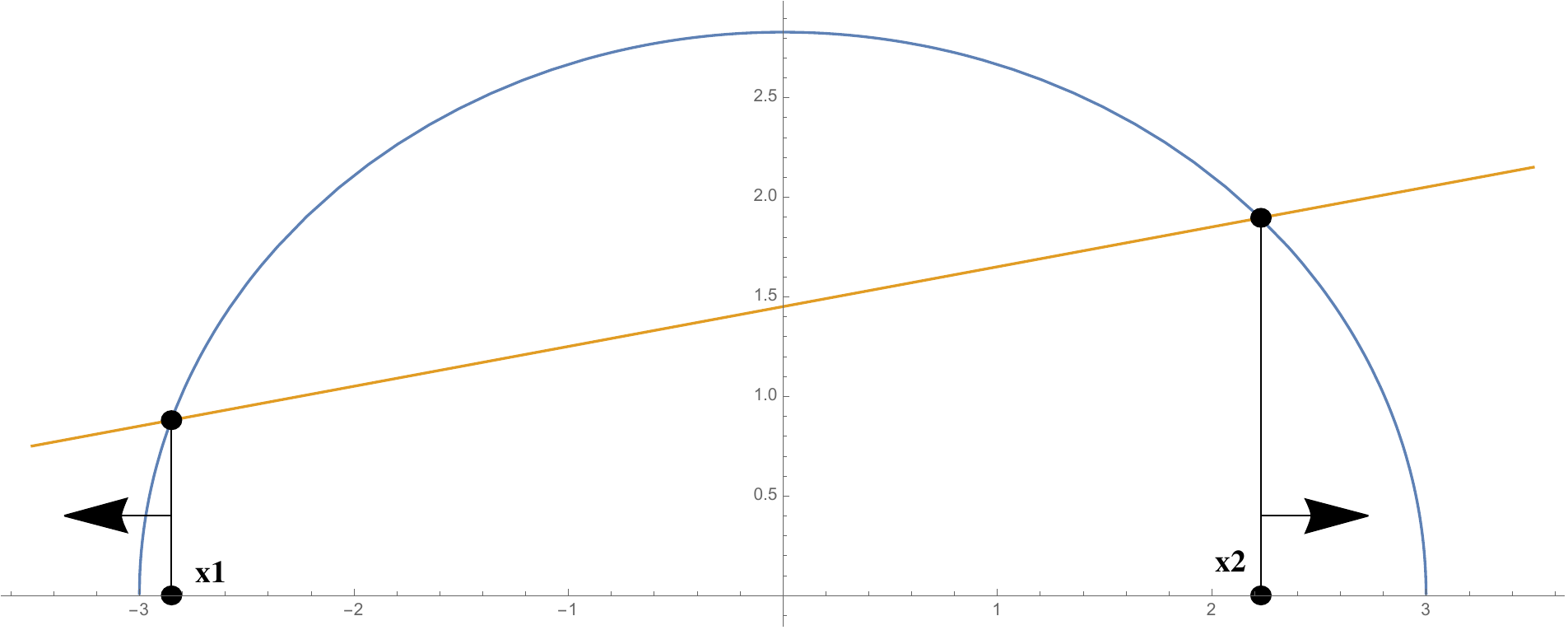} 
 \caption{As the ellipse expands, the vertical line $x=x_2$ moves  to the right faster than the line $x=x_1$ moves to the left.}
\label{F:velo}
\end{figure}

We conclude this section with a few comments.
\begin{itemize}
\item Assuming $b,m>0,$ an analysis of the second derivative of $g(t)=x_2(t)/x_1(t)$ indicates that $g(t)$ has a unique maximum. It follows that the runs of 1's trapped between two 0's  forms a unimodal sequence. This restricts the number of possible permutations, so it is certainly true that the bound $c_{n-1}=2F_{n+2}-2n$ is too big. It should be possible to get a better bound taking advantage of unimodality. But this effect does not appear for $n \leq 10$: the first ten values given for $b_n$ in Table~\ref{Ta:2vpline} coincide with the first ten values of $c_n=2F_{n+2}-2n.$

\item Another approach to the 2-vantage point problem in the plane uses  hyperbolas. In this formulation, we return to the original set-up of the 2-vantage point problem: $S$ is fixed and the vantage points move. Suppose $V_1$ and $V_2$ are the vantage points, then fix $V_1$ and let $V_2$ move. Choose two points $P_1,P_2\in S.$ When $V_2$ satisfies the equation $d(V_1,P_1)+d(V_2,P_1)=d(V_1,P_2)+d(V_2,P_2),$ we note that $P_1$ and $P_2$ will be tied in the order.

Now rewrite this equation as follows: $$d(V_2,P_1)-d(V_2,P_2)=d(V_1,P_2)-d(V_1,P_1).$$ But $d(V_1,P_2)-d(V_1,P_1)$ is a constant since $V_1,P_1,$ and $P_2$ are all fixed. So we are interested in the points $V_2$ where $d(V_2,P_1)-d(V_2,P_2)=c$ for some constant $c.$ The points that satisfy this equation form a hyperbola since the difference of the distances from $V_2$ to two fixed points is constant. But the regions determined by the hyperbolas that arise do not have the same properties as the hyperplane arrangements associated with perpendicular bisectors. In particular, it is possible for distinct regions to give the same order.

\end{itemize}

\section{Problems and conjectures}\label{S:prob}
We list several open problems that we believe deserve further study.

\begin{enumerate}
\item Gap filling: Given an integer $k$ satisfying $m(n,d)\leq k \leq M(n,d),$ we say that $k$ is \textit{achievable} if there is a configuration $S\subset \R^d$ of $n$ points so that the total number of orderings of the $n$ points is exactly $k,$ where the vantage point moves through $\R^d.$   When $d=1,$ we know that there are configurations that achieve $k$ orderings for any $k$ between the minimum and the maximum. This is false in dimensions $d>1,$ however.

\begin{prob}\label{Pr:gap}
Given $n, d>1,$ find all $k$ so that $k$ is achievable by an $n$-point configuration in $\R^d.$
\end{prob}

It is probably hopeless to answer Problem~\ref{Pr:gap} completely. Partial results along the lines of Cor.~\ref{C:gaps} should be possible; indeed, Theorem~\ref{T:gap1} holds in all dimensions. We also point out that the percentage of achievable values between the minimum and maximum from Table~\ref{Ta:gaps} shows that more than half of all possible values are achieved when $n \leq 8.$ It would be interesting to determine bounds for limit points for the sequence of achievable percentages in the plane:

\begin{prob}\label{Pr:gap2}
Let $r_n$ be the percentage of values between $m(n,2)$ and $M(n,2)$ that can be achieved by a configuration of $n$ points in the plane. Find upper and lower bounds for $r_n.$ In particular, does $r_n$ have a non-zero limit point?
\end{prob}

\item Minimum values: By Theorem~\ref{T:min}, we know the minimum value $m(n,d)=2n-2$ holds in $\R^d$ for all $d>1.$ But the configuration that achieves the minimum in $\R^d$ is 1-dimensional. Determining the minimum when we restrict to configurations whose affine spans are $d$-dimensional should be worth pursuing.

\begin{prob}\label{Pr:min}
Given $n,d>1,$ find the minimum number of orderings possible for $S\subset \R^d$ assuming the affine span of $S$ is $\R^d.$
\end{prob}

For instance, in the plane, if $S$ consists of the $n$ vertices of a regular $n$-gon, then $a_S(n,2)=2n,$ and it is easy to show this is best possible for configurations that span a plane. We would expect highly symmetric configurations to achieve values at or near the minimum in higher dimensions, too.
 \item Two vantage points: When we have two vantage points, we know very little when $d>1.$ %Concentrating on the plane, it would be interesting to determine any reasonable bounds for the maximum and minimum values possible for an $n$-point configuration.
 
 \begin{prob}\label{Pr:2pts}
 \begin{enumerate}
\item Let $m_2(n)$ and $M_2(n)$ be the minimum and maximum values for the number of orderings of an $n$-point set which are possible when two vantage points are allowed to move about the plane. Determine  $m_2(n)$ and $M_2(n).$
\item Determine asymptotic bounds for $m_2(n)$ and $M_2(n).$ Is it true that both $m_2(n)$ and $M_2(n)$ grow exponentially?
\end{enumerate}
 \end{prob}
 
 By Prop.~\ref{P:velo}, we know that $m_2(n)\leq \tau^{n+2},$ but this is an exponential upper bound on the minimum number of orderings, where $\tau$ is the golden mean. We have essentially no information about the maximum in this case.
 
\item More vantage points: When we allow two vantage points, more orderings are possible than with one (except for linear configurations). How many vantage points do we need to ensure that \textit{every} ordering of the points of $S$ can be realized by some placement of the vantage points?% It is clear [is it?] that if we increase the number of vantage points , then we will be able to achieve each of the $n!$ orderings of $S.$

\begin{prob}\label{Pr:factorial}
Let $S\subset \R^d$ with $|S|=n,$  and let $v(n,d)$ be the smallest number of vantage points so that the number of orderings of $S$ is $n!,$ where the $v(n,d)$ vantage points are free to move about $\R^d.$ Determine $v(n,d).$
\end{prob}

If $S \subset \R^d$ with $S=\{P_1,P_2,\dots,P_n\},$ then it should be possible to add $a_i$ vantage points at (or near) $P_i$ so that a specified ordering of $S$ can be achieved. One possible method for attacking this problem uses the fact that the Euclidean distance matrix (where the $(i,j)$ entry $d_{i,j}=d(P_i,P_j)$) is invertible \cite{m}. We can then solve a system of equations in the $a_i$ to achieve a specified order of $S.$ The $a_i$ may not be integers however, and they also may not be positive. But modifying this procedure may produce positive integer solutions for any desired order.

 \item Spherical arrangements: When a configuration $S\subset S^2$ of points on the sphere is contained in a hemisphere and also produces the maximum number of regions on a sphere, Theorem~\ref{T:double} gives a formula allowing us to determine the number of regions of the `doubled' configuration $S\cup (-S).$ (This is the configuration consisting of $S$ and the antipodes of each point of $S.$) The formula depends on the size of $S.$ Can this be generalized to other `doubled' configurations?
 
 \begin{prob}\label{Pr:double}
 Let $S \subset S^2$ be a collection of points contained in an open hemisphere of $S^2,$ and let $S\cup(-S)$ be the configuration obtained from $S$ by adding in all the antipodal points of $S.$ Is there a formula depending only on the size of $S$ and the number of regions determined by $S$ for the number of regions determined by $S\cup (-S)$?
 \end{prob}
 
  \end{enumerate}
  
  \section{Acknowledgements}  We thank Liz McMahon and Tom Zaslavsky for several helpful conversations. We also thank the anonymous referee for the proof idea given in item C following the proof of Theorem \ref{T:min}.


\begin{thebibliography}{99}


\bibitem{az} M. Aigner and G. Ziegler, ``Proofs from The Book'', Springer, Berlin, sixth edition, 2018. %See corrected reprint of the 1998 original [ MR1723092], Including illustrations by Karl H. Hofmann.

\bibitem{cov}  T. M. Cover, The number of linearly inducible orderings of points in d-space, \textit{SIAM J. Appl. Math.}, \textbf{15} (1967), 434--439.

\bibitem {dbe}  N. G. de Bruijn and P. Erd\H{o}s, On a combinatorial problem, \textit{Indagationes Math.}, \textbf{10} (1948),  421--423.


\bibitem{es}  P. Erd\H{o}s and R. Steinberg, Problems and Solutions: Advanced Problems: Solutions: 4065, \textit{Amer. Math. Monthly}, \textbf{51} (1944), 169--171.


\bibitem{gt} I. J. Good and T. N. Tideman, Stirling numbers and a geometric structure from voting theory, \textit{J. Combin. Theory Ser. A}, \textbf{23} (1977), 34--45.


\bibitem{j} R. E. Jamison, A survey of the slope problem, in ``Discrete geometry and convexity'',  Ann. New York Acad. Sci., \textbf{440} 34--51, New York Acad. Sci., New York, 1985.


\bibitem{m} C. A. Micchelli, Interpolation of scattered data: distance matrices and conditionally positive definite functions, in ``Approximation theory and spline functions'', NATO Adv. Sci. Inst. Ser. C Math. Phys. Sci., \textbf{136}, 143--145, Reidel,   Dordrecht, 1984.


\bibitem{pps} J. Pach, R. Pinchasi, and M. Sharir, Solution of Scott's problem on the number of directions determined by a point set in 3-space, \textit{Discrete Comput. Geom.}, \textbf{38} (2007), 399--441.


\bibitem{s} N. J. A. Sloane. The on-line encyclopedia of integer sequences, \url{https://oeis.org/}.

\bibitem {u} P. Ungar, $2N$ noncollinear points determine at least $2N$ directions, \textit{J. Combin. Theory Ser. A}, \textbf{33} (1982), 343--347.

\bibitem{z} T. Zaslavsky, Perpendicular dissections of space, \textit{Discrete Comput. Geom.}, \textbf{27} (2002), 303--351.


\end{thebibliography}
\end{document}